\documentclass[11pt]{article}

\usepackage{amsthm,amsmath,amsfonts,dsfont,graphicx,a4wide}

\renewcommand{\P}{\operatorname{P}}
\newcommand{\Cov}{\operatorname{Cov}}

\newcommand{\conver}{\mathop{\longrightarrow}}
\newcommand{\GQ}{{\mathds Q}}
\newcommand{\GN}{{\mathds N}}
\newcommand{\GR}{{\mathds R}}

\newcommand{\GZ}{{\mathds Z}}
\newcommand{\bn}{{\mathbf n}}
\newcommand{\bm}{{\mathbf m}}
\newcommand{\bN}{{\mathbf N}}
\newcommand{\bP}{{\mathbf P}}
\newcommand{\bR}{{\mathbf R}}
\newcommand{\bp}{{\mathbf p}}
\newcommand{\br}{{\mathbf r}}

\newcommand{\bk}{{\mathbf k}}
\newcommand{\bj}{{\mathbf j}}
\newcommand{\bi}{{\mathbf i}}

\newcommand{\bx}{{\mathbf x}}
\newcommand{\bt}{{\mathbf t}}
\newcommand{\bq}{{\mathbf q}}
\newcommand{\cE}{{\mathcal E}}
\newcommand{\cD}{{\mathcal D}}

\theoremstyle{plain}
\newtheorem{theo}{Theorem}[section]
\newtheorem{prop}[theo]{Proposition}
\newtheorem{lem}[theo]{Lemma}
\newtheorem{corol}[theo]{Corollary}
\theoremstyle{remark}
\newtheorem{defin}[theo]{Definition}
\newtheorem{rem}[theo]{Remark}

\begin{document}
\title{Directional phantom distribution functions\\ for~stationary random fields}

\author{Adam Jakubowski${}^{1}$\footnote{E-mail: adjakubo@mat.umk.pl}, Igor Rodionov${}^{2}$\footnote{E-mail: vecsell@gmail.com} \\ and Natalia Soja-Kukie\l a${}^1$\footnote{E-mail: natas@mat.umk.pl} \\[3mm]
${}^{1}$ Nicolaus Copernicus University, Toru\'n, Poland\\
	${}^2$ Steklov Mathematical Institute\\ of Russian Academy of Sciences, Moscow, Russian Federation 
	}

\date{}

\maketitle

\begin{abstract}
	We give necessary and sufficient conditions for the existence of a phantom distribution function for a stationary random field on a regular lattice. We also introduce a less demanding notion of a directional phantom distribution, with potentially broader area of applicability. Such approach leads to sectorial limit properties, a phenomenon well-known in limit theorems for random fields. An example of a stationary Gaussian random field is provided showing that the two notions do not coincide. Criteria for the existence of the corresponding notions of the extremal index and the sectorial extremal index are also given.
\end{abstract}

\noindent {\em Keywords:}
stationary random fields; extreme value limit theory; phantom distribution function; extremal index; Gaussian random fields

\noindent{\em MSClassification 2010:} 60G70, 60G60, 60G15.

\section{Introduction and announcement of results}
\subsection{Phantom distribution functions for sequences}

The notion of a phantom distribution function was introduced by O'Brien \cite{OBr87}. Let  $\{X_n:n\in\GZ\}$ be a stationary sequence with a marginal distribution function $F$ and partial maxima $M_n:= \max\{X_k: 1\leq k \leq n \}$, $n\in\GN$. We say that a distribution function $G$ is a {\em phantom distribution function for $\{X_n\}$}, if
\begin{equation*}\label{def_pdf_seq}
\sup_{x\in \GR} \left| P( M_n \leq x) - G(x)^n\right| \xrightarrow[n\to\infty]{} 0.
\end{equation*}
This means that $G$ completely describes asymptotic properties (in law) of partial maxima $\{X_n\}$. $G$ is also involved in description of asymptotics of higher order statistic of $\{X_n\}$ (see \cite{J93} and \cite{S-K17}).

If $G$ can be chosen in the form $G(x) = F^{\theta}(x)$, i.e. if
for some $\theta \in (0,1]$
\begin{equation*}\label{def_ei_seq}
\sup_{x\in \GR} \left| P( M_n \leq x) - P(X_0\leq x)^{\theta n}\right| \xrightarrow[n\to\infty]{} 0,
\end{equation*}
then, following Leadbetter \cite{Lea83}, we call $\theta$ the extremal index of $\{X_n\}$. The extremal index is a popular tool in the stochastic extreme value limit theory (see e.g. \cite{LLR83}). There exist, however, important classes of stationary sequences which admit a {\em continuous} phantom distribution function, while the notion of the extremal index is irrelevant in the description of the asymptotics of their partial maxima. This holds, for example, when Lindley's process has
subexponential innovations  \cite{Asmu98} or when the continuous 
target distribution of the random walk Metropolis algorithm has heavy tails \cite{RRSS06}.

Existence of a phantom distribution function is a quite common property. Doukhan et al. \cite[Theorem 6]{DJL15} show, that any $\alpha$-mixing sequence with {\em continuous} marginals admits a continuous phantom distribution function.
General Theorem 2, {\em ibid.}, asserts that a stationary sequence $\{X_n\}$ admits a continuous phantom distribution function if, and only if,   
there exists a sequence $\{v_n\}$ and $\gamma \in (0,1)$  such that 
\begin{equation}\label{e3}
\P(M_n\leq v_n)\to \gamma,
\end{equation}
and for each $T > 0$ the following {\bf Condition $\mathbf{B}_{T}(\{v_n\})$}  is fulfilled: 
\begin{equation}\label{e4a}
\sup_{p,q\in \GN,\atop \ p+q \leq T\cdot n}\left| \P\big(M_{p+q}\leq v_n\big) - \P\big(M_p\leq v_n\big) \P\big(M_q\leq v_n\big)\right| 
\to 0.
\end{equation}
Notice that {\bf Condition $\mathbf{B}_{T}(\{v_n\})$} can be satisfied even by non-ergodic sequences (see Theorem 4, {\em ibid.}). {\bf Condition $\mathbf{B}_{T}(\{v_n\})$} was introduced in \cite{Jak91}.

Another interesting issue is that there are ``user-friendly'' criteria of existence of a phantom distribution function for arbitrary (non-stationary) sequences - see \cite{Jak93} and \cite[Theorem 3]{JaTr18}. Such results are particularly useful in investigating Markov chains ``starting at the point".

\subsection{Phantom distribution functions for random fields}

As the previous section shows, the theory of phantom distribution functions for random {\em sequences} is essentially closed. It is therefore surprising that the corresponding theory of phantom distributions for {\em random fields over $\GZ^d$} is still far from being complete.

Let $\GZ^d$ be the $d$-dimensional lattice built on integers with the standard (coordinatewise) partial order $\leq$. Let $\{X_\bn : \bn\in\GZ^d\}$ be a $d$-dimensional {\em stationary} random field with a marginal distribution function $F$ and partial maxima defined for $\bj,\bn\in\GZ^d$ by the formulae 
\[M_{\bj,\bn} := \max \{ X_\bk  : \bj \leq \bk  \leq \bn\}, \ \text{ if }\  \bj\leq \bn,\qquad  M_{\bj,\bn} := -\infty,\ \text{ if }\ \bj \not\leq \bn.\]
It is also convenient to define 
\[
M_\bn:= M_{\mathbf{1}, \bn},\ \ \bn \in \GZ^d.\]
Of course, $M_\bn$ is of interest only if $\bn \in \GN^d$ (here and in the sequel we distinguish between $\GN= \{1,2,\ldots\}$ and
$\GN_0 = \{0\} \cup \GN$).

It seems that the first paper that mentions the notion of a phantom distribution function in the context of random fields is \cite{JS}.
Following this paper we will say that $G$ is a phantom distribution function for $\{X_\bn\}$, if
\begin{equation}\label{def_strong_pdf}
\sup_{x\in\GR}\left|P\left(M_\bn \leq x\right) - G(x)^{{\bn}^*} \right| \rightarrow 0, \text{ as $\bn \to \pmb{\infty}$  (coordinatewise)},
\end{equation}
where $\bn^* = n_1 \cdot n_2 \cdot \ldots \cdot n_d$, if $\bn = (n_1,n_2, \ldots, n_d)$.

Theorem 4.3 {\em ibid.} states that $m$-dependent random fields as well as moving maxima, moving averages and Gaussian fields satisfying Berman's condition admit a phantom distribution function in the above, strong sense. Another family of interesting examples, exploring the idea of a {\em tail field} in the context of the extremal index can be found in \cite{WuSa18}.

Note that (\ref{def_strong_pdf})  describes the asymptotic behavior of $M_{\bn}$ {\em regardless} of the way in which $\bn$ grows to $\pmb{\infty} = (\infty,\infty,\ldots,\infty)$. To make this statement precise, let us define a {\em monotone curve  in $\GN^d$} as a map $\pmb{\psi} : \GN \to \GN^d$ such that $\pmb{\psi}(n) \to \pmb{\infty}$, for $n=1,2,\ldots$ $\pmb{\psi} (n) \leq \pmb{\psi}(n+1)$ and $\pmb{\psi} (n) \neq \pmb{\psi}(n+1)$ (hence $\{\pmb{\psi}(n)^*\}$ is strictly increasing) and, as $n\to\infty$,
\begin{equation}\label{eq:crucial}
\frac{\pmb{\psi}(n)^*}{\pmb{\psi}(n+1)^*} \to 1.
\end{equation}

We will say that $G$ is a phantom distribution function for $\{X_\bn\}$ {\em along } $\pmb{\psi}$, if 
\begin{equation}\label{eq:along}
 \sup_{x\in \GR} \left| P ( M_{\pmb{\psi}(n)} \leq x ) - G(x)^{ \pmb{\psi}(n)^*}\right| \to 0, \text{ as $n \to \infty$}.
 \end{equation}
Any function $G$ satisfying (\ref{eq:along}) will be denoted by $G_{\pmb{\psi}}$.
Within such terminology we have the following
\begin{prop}\label{prop:zero}
	A stationary random field $\{X_\bn\}$ admits a continuous phantom distribution function $G$ if, and only if,  $G$ is a phantom distribution function for $\{X_\bn\}$  along {\em every} monotone curve.
	\end{prop}

Another consequence of (\ref{def_strong_pdf}) is that
if $x$ has the property that $G(x)^{{\bn}^*}$ is a ``good''
approximation of $P\left(M_{\bn} \leq x\right)$, then it is
equally good for all other points $\bm$ with $\bm^* = \bn^*$.
In other words, such $x$ is a function of the class
$L_k = \{\bn \in \GN^d\,;\,\bn^* = k\}$ rather, than of
$\bn$ alone. We formalize this observation by introducing the
notion of a {\em strongly monotone field of levels}. We will say
that $v_{(\cdot)} : \GN^d \to \GR^1$ is strongly monotone, if
 $v_{\bm} \leq v_{\bn}$ whenever ${\bm}^* \leq \bn^*$. 
This implies, in particular, that $v_{\bm} = v_{\bn}$, 
if ${\bm}^* ={\bn}^*$.

We are now able to give a multidimensional analog of \cite[Theorem 2]{DJL15}.  

\begin{theo}\label{theo1}
	Let $\{X_\bn : \bn\in\GZ^d\}$ be a stationary random field.
	Then $\{X_\bn\}$ admits a continuous phantom distribution function if, and only if, the following two conditions are satisfied.
	\begin{description}
		\item{\rm (i)} There exist $\gamma \in (0,1)$ and a {\em strongly monotone} field of levels $\{v_{\bn}\,;\, \bn \in \GN^d\}$  such that
		\[ P( M_{\bn} \leq v_{\bn}) \to \gamma, \text { as $\bn \to \pmb{\infty}$}.\]
		\item{\rm (ii)} For every monotone curve  $\pmb{\psi}$ and every $T > 0$ the following {\bf Condition $\mathbf{B}_T^{\pmb{\psi}}(\{v_{\pmb{\psi}(n)}\})$} holds.
	\end{description}	
	\begin{align*}
	\beta^{\pmb{\psi}}_T(n):=
	\max_{\mathbf{p}(1)+\mathbf{p}(2)\leq T \pmb{\psi}(n)}\Big|  P&\left(M_{\mathbf{p}(1)+\mathbf{p}(2)} \! \leq \! v_{\pmb{\psi}(n)}\right)
	 \\
	& - \prod_{\bi \in\{1,2\}^d} \!\!\!   P\left(M_{(p_1(i_1),p_2(i_2),\ldots,p_d(i_d))} \! \leq \! v_{\pmb{\psi}(n)}\right)\Big| \xrightarrow[n\to\infty]{} 0.
	\end{align*}
	(The quantities $\mathbf{p}(1)$ and $\mathbf{p}(2)$ under maximum take values in $\GN_0^d$).
\end{theo}
Condition $\mathbf{B}_T^{\pmb{\psi}}(\{v_{\pmb{\psi}(n)}\})$ looks complicated but it is based on a~simple idea. We shall illustrate it in the two-dimensional case. Notice that for $d=2$ we have
\begin{multline*}
\beta^{\pmb{\psi}}_T(n)= \max_{\mathbf{p}+\mathbf{q}\leq T \pmb{\psi}(n)}\Big| P\left(M_{\mathbf{p}+\mathbf{q}}\!\leq v_{\pmb{\psi}(n)}\right) -  \\
\quad P\left(M_{\mathbf{p}}\leq v_{\pmb{\psi}(n)}\right) P\left(M_{(p_1, q_2)}\leq v_{\pmb{\psi}(n)}\right) P\left(M_{(q_1,p_2)}\!\leq v_{\pmb{\psi}(n)}\right) 
P\!\left(M_{\mathbf{q}}\leq v_{\pmb{\psi}(n)}\right)\Big| 
\end{multline*}
and, moreover, by the stationarity,
\begin{eqnarray*}
	P\left(M_{(p_1, q_2)}\leq v_{\pmb{\psi}(n)}\right)&=& P\left(M_{(1,p_2+1),(p_1,p_2+q_2)}\leq v_{\pmb{\psi}(n)}\right),\\
	P\left(M_{(q_1,p_2)}\leq{v_{\pmb{\psi}(n)}}\right)&=& P\left(M_{(p_1+1,1),(p_1+q_1,p_2)}\leq{v_{\pmb{\psi}(n)}}\right),\\
	P\left(M_{\mathbf{q}}\leq v_{\pmb{\psi}(n)}\right)&=& P\left(M_{\mathbf{p}+\mathbf{1},\mathbf{p}+\mathbf{q}}\leq v_{\pmb{\psi}(n)}\right).
\end{eqnarray*}
It follows that if $\beta^{\pmb{\psi}}_T(n)\to 0$, as $n\to\infty$,  then $ P\left(M_{\mathbf{p}+\mathbf{q}}\leq v_{\pmb{\psi}(n)}\right)$  can be approximated by the product of the four probabilities for maxima over disjoint blocks, as in  Figure~\ref{fig_blocks}.

\begin{figure}[h] 
	\centering
	\setlength{\fboxsep}{10pt}
	\setlength{\fboxrule}{0.1pt}
	\fbox{\includegraphics[width=8cm]{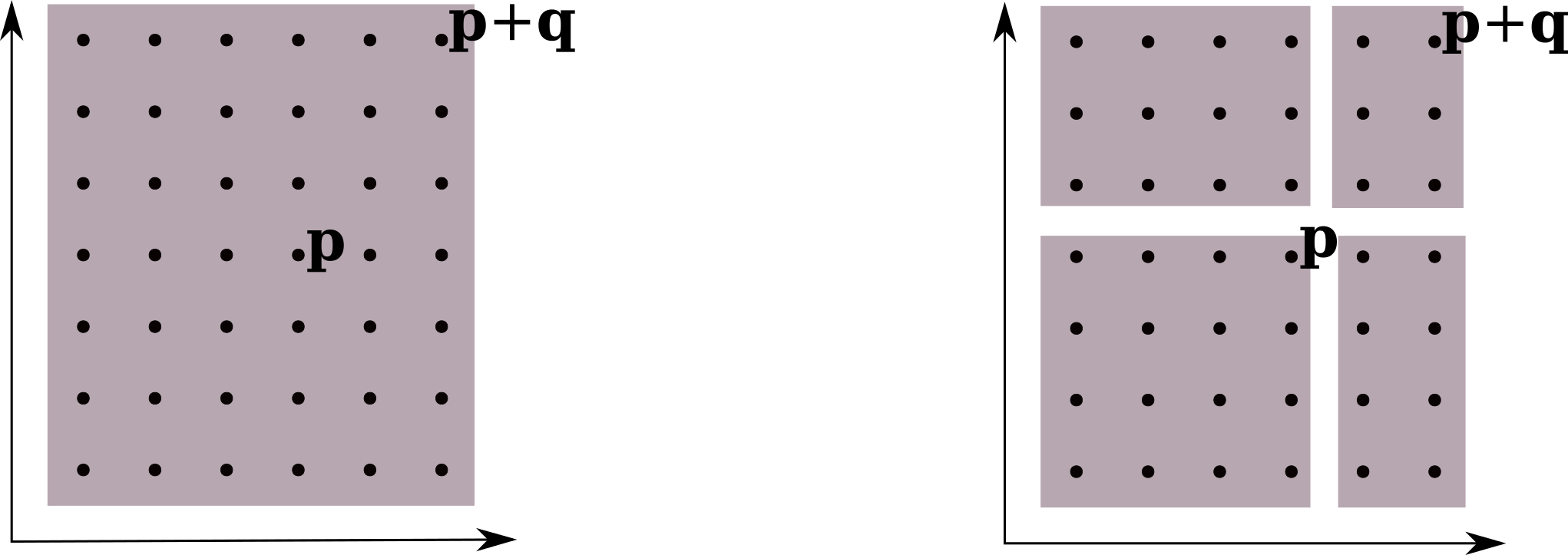}}
	\caption{Breaking probabilities into blocks as a consequence of Condition $\mathbf{B^{\pmb{\psi} }_T}(\{v_{\pmb{\psi}(n)}\})$, for $d=2$.} \label{fig_blocks}
\end{figure} 

By convention, if some coordinate of $\mathbf{p}$ or $\mathbf{q}$ is $0$, then $P\left(M_{\mathbf{p}+\mathbf{q}} \leq v_{\pmb{\psi}(n)}\right)$ breaks into smaller number of blocks (for $d =2$ into 2 or 1 block).

\begin{rem} By \cite[Theorem 4.3]{JS} models exhibiting local
	 dependence (like $m$-dependent or max-$m$-approximable
	  random fields) admit a continuous phantom distribution
	   function and so, by our Theorem  \ref{theo1}, 
	   satisfy Condition $\mathbf{B}_T^{\pmb{\psi}}(\{v_{\pmb{\psi}(n)}\})$.
\end{rem}

\begin{rem} Readers familiar with mixing conditions may not like the shape of Condition $\mathbf{B}_T^{\pmb{\psi}}(\{v_{\pmb{\psi}(n)}\})$ for there is no ``separation of blocks''. For example Leadbetter and Rootz\'en \cite{L-R98} investigate the asymptotics of maxima of stationary fields under {\em Coordinatewise mixing}, which involves separation of blocks. 	Ling \cite{Ling19} operates with Condition A1 (also involving separation of blocks), which is an adaptation of the well-known Condition D for sequences.
	Apart from the more complicated form of these conditions (that would be overhelming in $d$-dimensional considerations), they are essentially not easier in verification. We find the form of Condition $\mathbf{B}_T^{\pmb{\psi}}(\{v_{\pmb{\psi}(n)}\})$ very useful in theoretical consideration, for it reflects the intuition of breaking probabilities into product of probabilities over blocks and avoids technicalities.
	
	As a good example of how to check Condition $\mathbf{B}_T^{\pmb{\psi}}(\{v_{\pmb{\psi}(n)}\})$ (in one dimension) may serve Theorems 6-9 in \cite{DJL15}. 	
\end{rem}

\begin{rem}\label{R-rec}
	Suppose that $F$ is continuous. Choose $\gamma \in (0,1)$ and define the following field of levels:
	\[ v_{\bn} = \inf \{ x\,:\, P\left( M_{\bn} \leq x\right)  = \gamma\}.\]
	Then $\{v_{\bn}\}$ is non-decreasing, we have $P( M_{\bn} \leq v_{\bn}) \to \gamma$, but there is no reason to expect that it is strongly monotone. 
\end{rem}

\subsection{Directional and sectorial phantom distribution functions}

Remark \ref{R-rec} signalizes a serious difficulty
and suggests that the theory of phantom distribution functions
(and of the extremal index) in the sense of the strong
definition (\ref{def_strong_pdf}) is restricted to random fields with really
short-range dependencies (numerous examples of which are mentioned in the previous section).

It may happen that in some models another, weaker notion is 
more suitable. This is not an exceptional situation in the 
theory of random fields. For example, Gut \cite{GUT} gives
strong laws for i.i.d. sequences indexed by a {\em sector} 
and Gadidov \cite{GAD} deals with a similar framework for
$U$-statistics. Motivated by these examples we propose 
a new notion of a {\em directional } phantom distribution function.

Let $\{\pmb{\psi}(n)\}$ be a monotone curve. We define the class $\mathcal{U}_{\pmb{\psi}}$ of monotone curves, being a~kind of a~``neighbourhood'' of $\pmb{\psi}$, as follows. A monotone curve $\pmb{\varphi}$ belongs to $\mathcal{U}_{\pmb{\psi}}$ if and only if for some constant $C\geq 1$ and for almost all $n\in\GN$
\begin{equation*}\label{def_U_C}
\pmb{\varphi}(n)\in U(\pmb{\psi},C):=\bigcup_{j\in\GN} \prod_{i=1}^d [C^{-1} \psi_i(j), C \psi_i(j)].
\end{equation*}
An example of $U(\pmb{\psi},C)$ is shown in Figure \ref{fig_U}. 

\begin{figure}[h] 
	\centering
	\setlength{\fboxsep}{0pt}
	\setlength{\fboxrule}{0.1pt}
	\fbox{\includegraphics[width=5.5cm]{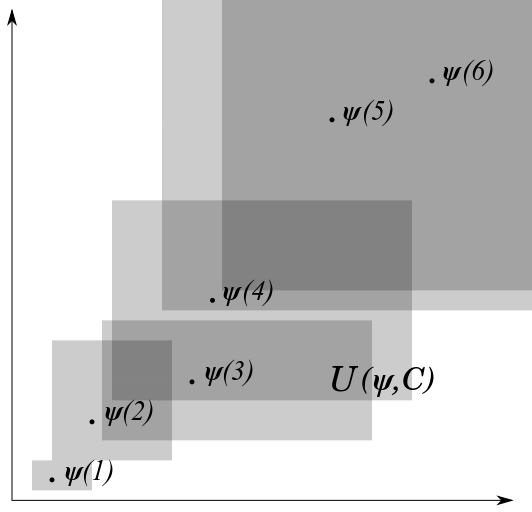}}
	\caption{The shaded area is the set $U(\pmb{\psi},C)\subset\GR ^2$ for $C=2$.}\label{fig_U}
\end{figure}

\begin{defin}
	Let $\{\pmb{\psi}(n)\}$ be a monotone curve. We will say that a distribution function $G$ is the $\pmb{\psi}$-directional phantom distribution function for $\{X_{\bn}\}$, if G is a phantom distribution function for $\{X_n\}$ along every monotone curve belonging to the set $\mathcal{U}_{\pmb{\psi}}$. We shall denote the $\pmb{\psi}$-directional phantom distribution function by $G_{\pmb{\psi}}$.
\end{defin}

Note that we already used the notation $G_{\pmb{\psi}}$ to denote the phantom distribution function {\em along} $\pmb{\psi}$. But there is no ambiguity. As we shall see in Theorem \ref{th:direct} below any  phantom distribution function {\em along} $\pmb{\psi}$ is automatically the $\pmb{\psi}$-directional phantom distribution function for $\{X_{\bn}\}$ and conversely.

\begin{rem} Let $\pmb{\Delta}(n)=(n,n,\ldots,n)$, $n\in\GN$, denote the diagonal map. 
	Observe that $\pmb{\varphi}$ belongs to $\mathcal{U}_{\pmb{\Delta}}$ if, and only if, $\varphi_1(n), \varphi_2(n),\ldots, \varphi_d(n)$ are of the same order, i.e., $1/C\leq \varphi_i(n)/\varphi_j(n)<C$ for some $C\geq 1$, all $i,j\in\{1,2,\ldots,d\}$ and almost all $n\in\GN$.
	It is natural to call  $G_{\pmb{\Delta}}$ a {\em sectorial} phantom distribution function.
\end{rem}

\begin{theo}\label{th:direct}
	Let $\{X_\bn : \bn\in\GZ^d\}$ be a stationary random field and let $\pmb{\psi}$ be a monotone curve.
	
	The following statements {\rm (i)-(iii)} are equivalent.
		\begin{description}
		\item{\rm (i)} $\{X_\bn\}$ admits a continuous phantom distribution function along $\pmb{\psi}$.
		\item{\rm (ii)} $\{X_\bn\}$ admits a continuous $\pmb{\psi}$-directional phantom distribution function. 
		\item{\rm (iii)} There exist $\gamma \in (0,1)$ and a {\em non-decreasing} sequence of levels $\{v_{\pmb{\psi}(n)}\}$, $n \in \GN$, such that
		\begin{equation}\label{eq:mainbis} 
		P( M_{\pmb{\psi}(n)} \leq v_{\pmb{\psi}(n)}) \to \gamma, \text { as $n \to \infty$},
		\end{equation}
		and for every $T > 0$ Condition $\mathbf{B}_T^{\pmb{\psi}}(\{v_{\pmb{\psi}(n)}\})$ holds.
	\end{description}	
\end{theo}

\begin{rem} We have not yet addressed the question that is basic 
	for this section: is there any model for which there exists a sectorial phantom distribution function while there is no global phantom distribution function? The answer is yes, and the example is given in the next section.
\end{rem}

\subsection{Example}\label{sec:example}
\subsubsection{The random field}
First we shall construct two characteristic functions $\eta_1(\theta)$ and $\eta_2(\theta)$ on $\GR^1$ using Polya's recipe (see \cite{Fell70}). The graph of $\eta_1$ over $\GR^+$ is a polygon connecting points: 
\[\big(0,1\big), \big(1, \gamma_1\big(27 \frac{\ln \big(\ln 27\big)}{\ln 27} - 26 \frac{\ln \big(\ln 28\big)}{\ln 28}\big)\big), \big(28, \gamma_1\frac{\ln \big(\ln 28\big)}{\ln 28}\big), \big(29, \gamma_1\frac{\ln \big(\ln 29\big)}{\ln 29}\big), 
\ldots,\]
while the graph of $\eta_2$ over $\GR^+$ is defined using a different sequence of points:
\[\big(0,1\big), \big(1, \gamma_2 \big( \frac{2}{\ln 2} - \frac{1}{\ln 3}\big)\big),  \big(3,\gamma_2 \frac{1}{\ln 3}\big), 
\big(4, \gamma_2 \frac{1}{\ln 4}\big),\ldots.\]
The graphs of $\eta_1$ and $\eta_2$ over $\GR^-$ are obtained by reflection.
The positive numbers $\gamma_1$ and $\gamma_2$ satisfy 
\begin{equation}
 \gamma_1 > 1/4,\quad \gamma_1\big(27 \frac{\ln \big(\ln 27\big)}{\ln 27} - 26 \frac{\ln \big(\ln 28\big)}{\ln 28}\big) <  \gamma_2 \big( \frac{2}{\ln 2} - \frac{1}{\ln 3}\big) < \frac{1 - 2 \gamma_1}{1 + 2\gamma_1}.
  \label{gammas}
 \end{equation}
The reader may verify that such numbers $\gamma_1,\gamma_2$ do exist and that the corresponding functions $\eta_1$ and $\eta_2$ satisfy Polya's criterion. Therefore both $\{\eta_1(i)\}_{i\in\GZ}$ and $\{\eta_2(j)\}_{j\in\GZ}$ are positively defined. It follows that 
\begin{equation}
\label{ex:e1}
 r_{ij} = \eta_1(i)\eta_2(j)
 \end{equation}
is a covariance function on $\GZ^2$. This function satisfies 
\begin{equation}\label{ex:e2}
\delta := \sup_{(i,j)\in \mathbb{Z}^2\backslash \{(0,0)\}} r_{ij} < \frac{1 - 2\gamma_1}{1 + 2\gamma_1} < \frac{1}{3},
\end{equation}
and for $i$ and $j$ with sufficiently large absolute values  we have  
\begin{equation}
r_{ij} = \gamma_1 \gamma_2 \frac{\ln\ln |i|}{\ln |i|} \frac{1}{\ln |j|}.\label{r}
\end{equation}

Let $\mathbf{X} = \{X_{(i,j)}, (i,j)\in\mathbb{Z}^2\}$ be a Gaussian stationary random field  with mean zero, unit variance and covariance function $E X_{(i,j)} X_{(0,0)} = r_{ij}.$

\subsubsection{$\Phi$ is a sectorial phantom distribution function}
We shall prove that 
\begin{equation}
\sup_{x\in \mathbb{R}}\left|P(M_{\mathbf{n}} \leq x) - \Phi(x)^{n^2}\right| \to 0, \text{ as } n\to\infty,\label{toprove1}
\end{equation}
where $\mathbf{n} = (n,n) = \pmb{\Delta}(n)$, $M_{\mathbf{n}} = \max_{(i,j) \in [1,n]\times[1,n]} X_{(i,j)}$ and $\Phi(x)$ is the distribution function of a standard normal random variable. Applying  Theorem \ref{th:direct} we will conclude that $\Phi$ is a $\pmb{\Delta}$-directional (or sectorial) phantom distribution function for  $\mathbf{X}$.

As usually, in order to prove (\ref{toprove1}) it is sufficient to show that for every $c > 0$
 \begin{equation}\label{eq:equiv}
  P\big(M_{\mathbf{n}} \leq u_{\mathbf{n}}(c)\big) = \Phi\big(u_{\mathbf{n}}(c)\big)^{n^2} + o(1),
  \end{equation}
 where levels $\{u_{\mathbf{n}}(c)\}$ are such that $n^2(1 - \Phi(u_{\mathbf{n}}(c))) \to c$. Note that  for  $n$ large enough
 \begin{equation}\label{u_n1}\exp\left(- \frac{(u_{\mathbf{n}}(c))^2}{2}\right) = \frac{\sqrt{2\pi}c u_{\mathbf{n}}(c)}{n^2} (1+o(1)) \leq \frac{2 \sqrt{\pi} c u_{\mathbf{n}}(c)}{n^2} = K(c) \frac{u_{\mathbf{n}}(c)}{n^2}\end{equation}
 and that 
 \begin{equation}\label{u_n1_bis}
u_{\mathbf{n}}(c) \sim \sqrt{4 \ln n},\quad \text{ as $n\to\infty$.}
 \end{equation}
We have by Berman's inequality for Gaussian stationary sequences (\cite[Corollary 4.2.4]{LLR83})
\begin{equation}\begin{split}
\big|P\big(M_{\mathbf{n}} \leq u_{\mathbf{n}}(c)\big) &- \Phi(u_{\mathbf{n}}(c))^{n^2}\big|\\
&\leq L(\delta) \sum_{(i,j), (k,l) \in \{1,2,\ldots,n\}^2 \atop (i,j)\neq (k,l)}\Cov\big(X_{(i,j)},X_{(k,l)}\big) \exp\big(-\frac{(u_{\mathbf{n}}(c))^2}{1 + r_{i-k,j-l}}\big)\\
&\leq 4L(\delta) n^2 \sum_{0\leq i, j\leq n \atop (i,j)\neq (0,0)}r_{ij} \exp\left(-\frac{(u_{\mathbf{n}}(c))^2}{1 + r_{i,j}}\right),
\end{split}
\label{berman1}
\end{equation}
where $L(\delta)$ is a constant depending only on $\delta$ and we have used the stationarity and the fact that $r_{ij} > 0$, $i,j\in \GZ$.
Repeating the steps of the proof of  \cite[Lemma 4.3.2]{LLR83}, choose $\alpha,$ $0<\alpha<\frac{1 - 3\delta}{1+\delta},$ (see (\ref{ex:e2})) and split the sum in the last line of (\ref{berman1}) in two parts $\Sigma_1(n) = \sum_{(i,j)\in A_n}$ and $\Sigma_2(n) = \sum_{(i,j)\in B_n},$ where 
\[A_n = \{\lceil n^\alpha \rceil,\ldots, n\}\times \{\lceil n^\alpha\rceil,\ldots,  n\} \text{ and } B_n = \{0,1,\ldots, n\}^2\setminus(A_n \cup \{\mathbf{0}\}).\]

First let us find the asymptotics of the part involving $\Sigma_2(n)$. We have for large $n$ \begin{align*}
 4L(\delta) n^2 \Sigma_2(n) &\leq 4 L(\delta) n^2 \big(2 n^{1+\alpha} -  \big(\lceil n^{\alpha}\rceil - 1) ^2\big)\exp\left(-\frac{(u_{\mathbf{n}}(c))^2}{1+\delta}\right) \\
&\leq 8 L(\delta) K(c)^{\frac{2}{1+\delta}} n^{3+\alpha} \left(\frac{u_{\mathbf{n}}(c)}{n^2}\right)^{\frac{2}{1+\delta}}\tag*{by (\ref{u_n1})}\\
&\sim 8 L(\delta) K(c)^{\frac{2}{1+\delta}} (4 \ln n)^{\frac{1}{1+\delta}} n^{\alpha + 3 - \frac{4}{1+\delta}} \to 0, \tag*{by (\ref{u_n1_bis}) and the choice of $\alpha$.}
\end{align*}

Next, let us notice that for $i, j \geq \lceil n^{\alpha}\rceil$ and $n$ large enough
\[ r_{ij} \leq \frac{\ln (\ln n^{\alpha})}{\big(\ln n^{\alpha}\big)^2} \leq \big(2/\alpha^2\big) \frac{\ln\ln n}{(\ln n)^2}. \]
Therefore, setting $\delta_n' = \sup_{i,j\in A_n} r_{i,j}$ and using (\ref{u_n1_bis}) we obtain that 
$ \delta_n' (u_{\mathbf{n}}(c))^2 \to 0$, as $n\to\infty$. Keeping this relation in mind we can   
proceed as follows.
\begin{align*}
4L(\delta) n^2 \Sigma_1(n) &=  4L(\delta) n^2 \sum_{(i,j)\in A_n} r_{ij} \exp\left(-\frac{(u_{\mathbf{n}}(c))^2}{1+r_{ij}}\right)\\ &= 4L(\delta) n^2 \exp(-(u_{\mathbf{n}}(c))^2) 
\sum_{(i,j)\in A_n} r_{ij} \exp\left(\frac{(u_{\mathbf{n}}(c))^2 r_{ij}}{1+r_{ij}}\right) \\
&\leq 4L(\delta) (K(c))^2 n^2    \Big(\frac{u_{\mathbf{n}}(c)}{n^2}\Big)^2 
n^2 \delta_n' \exp\big(\delta_n^\prime (u_{\mathbf{n}}(c))^2\big) \tag*{by (\ref{u_n1})}\\
&= 4L(\delta)(K(c))^2 \delta_n^\prime (u_{\mathbf{n}}(c))^2 \exp(\delta_n^\prime (u_{\mathbf{n}}(c))^2) \to 0,\ \ \text{ as $n\to\infty$}.
\end{align*}

\subsubsection{There is no global phantom distribution function}
 
 Let us consider the monotone curve
 \[ \pmb{\psi}(n) = \big( \lfloor n/\ln n\rfloor, \lfloor\ln n\rfloor\big), n\in \GN.\]
 By Proposition \ref{prop:zero}, it is enough to show that $\Phi$ {\em is not} a phantom distribution function for $\{X_{(i,j)}, (i,j)\in\mathbb{Z}^2\}$ along $\pmb{\psi}$. Notice that the desired property is in agreement with the statement of \cite[Theorem 6.5.1]{LLR83}, for we have 
\[ \pmb{\psi}(n)^* \sim n \text{\ \ and \ \ } r_{\pmb{\psi}(n)}\ln n \to \gamma_1 \gamma_2 >0.\]
The structure of random variables $M_{\pmb{\psi}(n)}$ is however more complicated than just partial maxima of a stationary Gaussian sequence and therefore we have to perform carefully all computations.

We will show first that 
\begin{equation}
\label{eq:prop2}
\sup_{x\in\GR} \big| P\big( M_{\pmb{\psi}(n)} \leq x\big) - 
P\big( \widetilde{M}_n \leq x\big)\big| \to 0, \text{ as } n\to\infty,
\end{equation}
where for each $n\in\GN$ $\widetilde{M}_n$ is the maximum  of $\pmb{\psi}(n)^*$ standard normal random variables $\xi_1, \xi_2,\ldots, \xi_{\pmb{\psi}(n)^*}$ with $\rho_n = \text{cov}(\xi_i, \xi_j) = \frac{\gamma_1\gamma_2}{\ln n},$ $i\neq j$.
As in the case of (\ref{toprove1}),  we have to prove that 
\[ P\big(M_{\psi(n)} \leq w_n(c)\big) = P\big(\widetilde{M}_{n} \leq w_n(c)\big) + o(1),\] 
for sequences of levels $\{w_n(c)\}$ such that $P\big(\widetilde{M}_{n} \leq w_n(c)\big) \to c \in (0,1)$. Later we shall show that $\{w_n(c)\}$ satisfies 
\begin{equation}\label{eq:wuen}
\exp\Big( -\frac{w_n(c)^2}{2}\Big) \leq K'(c) \frac{w_n(c)}{n}\ \text{ \ and \ } w_n(c) \sim \sqrt{2 \ln n}.
\end{equation}
 
 By virtue of \cite[Theorem 4.2.1]{LLR83}, we have 
 \begin{equation}
 \begin{split}
 \Big|P\big(M_{\psi(n)} \leq &w_n(c)\big) - P\big(\widetilde{M}_{n} \leq w_n(c)\big)\Big|\\ &\leq 4 L(\delta) n \sum_{(i,j)\in D_n} |r_{ij} - \rho_n| \exp\left( - \frac{(w_n(c))^2}{1 + {\omega}_{ij}}\right),\label{berman2}
 \end{split}
 \end{equation}
  where $D_n = \{(i,j): 0\leq i\leq \frac{n}{\ln n}, 0 \leq j \leq \ln n\}\setminus\{(0,0)\}$ and ${\omega}_{ij} = \max\{r_{ij}, \rho_n\} = r_{ij}$ on $D_n$. 
   Let us split the set of indices $D_n$ in three smaller parts, $D_n = D_n^{(1)} \sqcup D_n^{(2)} \sqcup D_n^{(3)}$, where $D_n^{(1)} =
 \{(i,j): 0\leq i\leq n^\alpha, 0\leq j\leq \ln n\}\setminus\{(0,0)\},$ $D_n^{(2)}= \{(i,j): n^\alpha< i\leq \frac{n}{\ln n}, 0\leq j\leq (\ln n)^\beta\}$ and $D_n^{(3)} = \{(i,j): n^\alpha < i\leq \frac{n}{\ln n}, (\ln n)^\beta< j\leq \ln n\}$, where the parameters $\alpha$ and $\beta$ will be chosen later.
 
 By (\ref{ex:e2}) we have $\delta < (1-2\gamma_1)/(1 + 2\gamma_1)$, or, equivalently, $2\gamma_1 < (1-\delta)/(1+\delta)$. So we can find $\alpha$ satisfying
 \[ 2\gamma_1 < \alpha < \frac{1 - \delta}{1+\delta}.\]
 From (\ref{eq:wuen}) we have, as $n\to\infty$,
 \begin{align*}
 n \sum_{(i,j)\in D_n^{(1)}} |r_{ij} &- \rho_n| \exp\Big( - \frac{(w_n(c))^2}{1 + r_{ij}}\Big) \leq  n\, n^\alpha \ln n \exp\left( - \frac{(w_n(c))^2}{1 + \delta}\right)  \\
 &\leq (K'(c))^{\frac{2}{1+\delta}} n^{\alpha+1} \ln n \left(\frac{w_n(c)}{n}\right)^{\frac{2}{1+\delta}} \sim  (\sqrt{2}K'(c))^{\frac{2}{1+\delta}}
 n^{\alpha - \frac{1 - \delta}{1+\delta}} (\ln n)^{\frac{2+\delta}{1+\delta}} \to 0. 
 \end{align*}
 
 Estimation of the term related to the sum over $(i,j)\in D_n^{(2)}$ is a bit more challenging. For indices $(i,j)\in D_n^{(2)}$ we have $|r_{ij} - \rho_n| \leq r_{ij} \leq \frac{\gamma_1}{\alpha} \frac{\ln\ln n}{\ln n} =: \delta_n$. Therefore we obtain 
 \begin{align}
 n \sum_{(i,j)\in D_n^{(2)}} |r_{ij} - \rho_n| &\exp\Big( - \frac{(w_n(c))^2}{1 + r_{ij}}\Big) \nonumber\\
 &\leq \frac{\gamma_1}{\alpha} n \frac{n}{\ln n} (\ln n)^\beta \frac{\ln\ln n}{\ln n} \exp\left(-\frac{(w_n(c))^2}{1 + \delta_n}\right) \nonumber \\
 &\leq \frac{\gamma_1}{\alpha} (K'(c))^2 n^2 (\ln n)^{\beta-2} \ln\ln n \left(\frac{\sqrt{2 \ln n}}{n}\right)^2 n^{2\delta_n} \nonumber \\
 &= \frac{\gamma_1}{\alpha} (K'(c))^2 (\ln n)^{\beta-1}
 \ln\ln n \exp\left(2\frac{\gamma_1}{\alpha} \frac{\ln\ln n}{\ln n} \ln n\right)\nonumber \\
 & = \frac{\gamma_1}{\alpha} (K'(c))^2 (\ln n)^{\beta + 2\gamma_1/\alpha -1} \ln\ln n.\label{C2}
 \end{align}
 Because $\gamma_1 < \alpha/2$, we can find positive $\beta$ satisfying the inequality $\beta + 2\gamma_1/\alpha -1 < 0.$ For such $\beta$ the expression in  (\ref{C2}) tends to $0$. 
 
 It remains to show that the term related to the sum over  $(i,j)\in D_n^{(3)}$ vanishes as $n\to\infty$.
 Denote $\delta_n' = \max_{(i,j)\in D_n^{(3)}} r_{ij}$ and notice that $\delta_n'\leq \frac{\gamma_1\gamma_2}{\alpha\beta}/\ln n.$ We need a special decomposition.
 \begin{align*}
 n \sum_{(i,j)\in D_n^{(3)}} &|r_{ij} - \rho_n| \exp\left( - \frac{(w_n(c))^2}{1 + r_{ij}}\right) 
 \leq  n \exp\left( - \frac{(w_n(c))^2}{1 + \delta_n'}\right) \sum_{(i,j)\in D_n^{(3)}} \big(r_{ij} - \rho_n\big) \\
 &= \Big\{ \frac{n^2}{\ln n}\exp\left( - \frac{(w_n(c))^2}{1 + \delta_n'}\right)\Big\} \cdot \Big\{\frac{\ln n}{n}\sum_{(i,j)\in D_n^{(3)}} \big(r_{ij} - \rho_n\big) \Big\} = I_1(n) \cdot I_2(n).\label{last}
 \end{align*}
Using  (\ref{eq:wuen}) we obtain the boundedness of $\{I_1(n)\}$.
 \begin{align*}
 I_1(n)  = \frac{n^2}{\ln n}\exp\Big( - \frac{(w_n(c))^2}{1 + \delta_n'}\Big)
  &\leq (K'(c))^2 \frac{n^2}{\ln n} \left(\frac{w_n(c)}{n}\right)^2 \left(\frac{n}{w_n(c)}\right)^{2\frac{\gamma_1\gamma_2}{\alpha\beta}/\ln n} \\ 
  & \sim (K'(c))^2 \frac{n^2}{\ln n} \frac{2\ln n}{n^2}e^{2\frac{\gamma_1\gamma_2}{\alpha\beta}}(1+ o(1)) = O(1).
 \end{align*}
 
 We will conclude the proof of (\ref{eq:prop2}) by showing that $I_2(n) \to 0$ as $n\to\infty$. We have
 \begin{align*}
 \frac{\ln n}{n}\sum_{(i,j)\in D_n^{(3)}} \big(r_{ij} &- \rho_n\big) = \frac{\ln n}{n}\sum_{(i,j)\in D_n^{(3)}} r_{ij} - \frac{\ln n}{n}\left(\frac{n}{\ln n} - n^\alpha\right)\left(\ln n - (\ln n)^\beta\right) \rho_n\\
 &= \gamma_1\gamma_2\frac{ \ln n}{n} \Big(\sum_{i=n^\alpha}^{n/\ln n} \frac{\ln\ln i}{\ln i}\Big) \Big(\sum_{j=(\ln n)^\beta}^{\ln n} \frac{1}{\ln j}\Big) - \gamma_1\gamma_2(1 + O((\ln n)^{\beta - 1})).
 \end{align*}
We shall estimate the two sums appearing above. By integration by parts we have for $ 1 < a < b$
\[ \int_a^b \ln t \frac{e^t}{t} \,dt \leq \frac{a}{a-1} \ln b \frac{e^b}{b}, \ \text{ and } \ \ \int_a^b \frac{e^t}{t} \,dt \leq \frac{a}{a-1}\frac{e^b}{b}.\]
Therefore 
 \begin{align*}
 \sum_{i=n^\alpha}^{n/\ln n} \frac{\ln\ln i}{\ln i} &\leq \int_{n^{\alpha/2}}^{n/\ln n} \frac{\ln \ln y}{\ln y} dy = \int_{\alpha \ln n/2}^{\ln n - \ln \ln n} \ln t\frac{e^t}{t}\,dt \\ &\leq \frac{\alpha \ln n/2}{\alpha \ln n/2 -1} \ln ( \ln n - \ln \ln n)\frac{e^{\ln n - \ln \ln n}}{\ln n - \ln \ln n} = \frac{n \ln\ln n}{(\ln n)^2} \left(1 + O\bigr(\frac{\ln\ln n}{\ln n}\bigl)\right).
 \end{align*}
 Similarly
\begin{align*} \sum_{j=(\ln n)^\beta}^{\ln n} \frac{1}{\ln j} &\leq \int_{(\ln n)^{\beta/2}}^{\ln n} \frac{1}{\ln y} dy = \int_{\beta \ln\ln n/2}^{\ln\ln n} \frac{e^t}{t} dt \\
&\leq \frac{\beta \ln\ln n/2}{\beta \ln\ln n/2 - 1} \frac{e^{\ln\ln n}}{\ln\ln n} = \frac{\ln n}{\ln\ln n}\left(1 + O\bigl(\frac{1}{\ln\ln n}\bigr)\right).
 \end{align*}
Finally we get
 \[I_2(n) \leq \gamma_1\gamma_2 \Big\{\Big(1 + O\bigr(\frac{1}{\ln\ln n}\bigl)\Big)\left(1 + O\bigr(\frac{\ln\ln n}{\ln n}\bigl)\right) - (1 + O((\ln n)^{\beta - 1}))\Big\} \conver 0, \ \text{ as $n\to \infty$}.\]

To complete the proof of (\ref{eq:prop2}) we have to verify
(\ref{eq:wuen}).

\begin{prop}\label{prop:emtylda}
There exists a continuous strictly increasing distribution function $H$ such that for every $x\in \GR$
\[ P\big( a_n\big(\widetilde{M}_n - b_n\big)\leq x\big) \to H(x),\]
where 
\[a_n = \sqrt{2\ln n},\quad b_n = \sqrt{2\ln n} - \frac{\ln\ln n + \ln (4\pi)}{2\sqrt{2\ln n}}, \ \ n\in \GN.\]

For each $c\in (0,1)$, let $x= x(c)$ be such that $H(x) = c$
and let 
$y_n(c) = x(c)/a_n + b_n$.

If $P\big( \widetilde{M}_n \leq w_n(c)\big) \to c \in (0,1)$, then $| w_n(c) - y_n(c)| = o(1/\sqrt{\ln n})$ and  $\{w_n(c)\}$ satisfies (\ref{eq:wuen}).
\end{prop}
\begin{proof}
	The proof of the first part of the proposition coincides, in fact, with a part of the proof of \cite[Theorem 6.5.1]{LLR83} (see also \cite{MiYl75}). But these results deal basically with {\em partial maxima of stationary sequences} and are not formulated in the scheme of triangular arrays, as is required by our setting. Therefore we provide here a complete argument.
	
	We may and do assume that $\pmb{\psi}(n)^* = n$. By the definition, $\widetilde{M}_n$ is equal in law to 
	$  \sqrt{1 - \rho_n}\,\widehat{M}_n+ \sqrt{\rho_n}\, \zeta$,
	where $\widehat{M}_n$ is the maximum of a sequence of $n$ independent standard normal random variables and $\zeta$ is standard normal independent of $\widehat{M}_n$. We thus obtain
	\begin{align*} P\big(a_n\big(\widetilde{M}_{n} - b_n\big)\leq x\big) &= P\Big( \sqrt{1 - \rho_n}\,\widehat{M}_n+ \sqrt{\rho_n}\, \zeta \leq x/a_n + b_n\Big)\\ 
	&= \int_{-\infty}^{\infty} P\Big(\widehat{M}_n \leq (1 - \rho_n)^{-1/2} \big(x/a_n + b_n - \sqrt{\rho_n} z\big) \Big) \,\varphi(z)\, dz \\
	& =\int_{-\infty}^{\infty} \Big(\Phi \big( (1 - \rho_n)^{-1/2} \big(x/a_n + b_n - \sqrt{\rho_n} z\big) \big)\Big)^n\, \varphi(z)\, dz \\
	&\conver \int_{-\infty}^{+\infty}\exp(-\exp(-x -\gamma_1\gamma_2 + \sqrt{2\gamma_1\gamma_2} z))\varphi(z)\,dz = : H(x),
	\end{align*}
	  because (see the proof of \cite[Theorem 6.5.1]{LLR83})
	\[(1 - \frac{\gamma_1\gamma_2}{\ln n})^{-1/2} (x/a_n + b_n - \sqrt{\frac{\gamma_1\gamma_2}{\ln n}} z) = \frac{x + \gamma_1\gamma_2 - \sqrt{2\gamma_1\gamma_2} z}{a_n} + b_n + o((a_n)^{-1}).\]
	 
Assume that $P\big( \widetilde{M}_n \leq w_n(c)\big) \to c \in (0,1)$. Consider levels $y_n(c) = x(c)/a_n + b_n$. Let $x' < x(c) < x^{\prime\prime}$. We have eventually
\[ \frac{x' - x(c)}{a_n} = x'/a_n + b_n - y_n(c) \leq w_n(c) - y_n(c) \leq x^{\prime\prime}/a_n + b_n - y_n(c) = \frac{x^{\prime\prime} - x(c)}{a_n}.\]
Because $x'$ and $x^{\prime\prime}$ can be chosen arbitrarily close, 
\[  | w_n(c) - y_n(c)| = o((a_n)^{-1}).\]
This clearly implies (\ref{eq:wuen}).
\end{proof}

Given (\ref{eq:prop2}), it is not difficult to prove that $\Phi(x)$ {\em is not } a phantom distribution function for 
$\{ X_{(i,j)}\}$ along $\pmb{\psi}$. Because $H(x)$ does not coincide with the Gumbel standardized distribution $H_0$,  we have $H_0(x_0) \neq H(x_0)$ for some $x_0$.  And we have proved
that $
P\big( \widetilde{M}_n \leq x_0/a_n + b_n\big) \to H(x_0)$, while we know that $\Phi(x_0/a_n + b_n)^n \to H_0(x_0)$.

\subsection{Extremal indices}

We will use the results of the previous sections to provide a complete theory of the extremal index for random fields. Recall that $F$ is the marginal distribution function of $X_{\bn}$.

\begin{defin}\label{def:index}
	We say that $\theta\in(0,1]$ is {\em the extremal index} for $\{X_\bn\}$, if the function $G$ given by $G(x):=P(X_\mathbf{0}\leq x)^{\theta}$, $x\in\GR $, is a phantom distribution function for $\{X_\bn\}$.
	
	If $G(x):=P(X_\mathbf{0}\leq x)^{\theta}$, for some $\theta\in(0,1]$, is a sectorial distribution function for $\{X_\bn\}$, then we say that $\theta$ is {\em the sectorial extremal index} for  $\{X_\bn\}$. 
\end{defin}

\begin{rem}
	This definition of the (global) extremal index is taken from  \cite{JS}. 
We note that a ``more classical'' definition of the (global) extremal index for random fields was proposed in \cite{CHOI}, see also  
\cite{Tu06} and \cite{FePe08}. These papers, however, did not bring conclusive results. For instance, the 
formula for calculating the extremal index proposed in \cite{FePe08} does not work for a simple  $1$-dependent random field given in \cite[Example 5.5]{JS}.

Examples of calculation of the global extremal index for a variety of random fields on the lattice $\GZ^d$ can be found in \cite{BaTa14} (moving averages and moving maxima), \cite{JS} (models with local dependence) and \cite{WuSa18} (regularly varying random fields). Some related work for Gaussian random fields is given in \cite{Ling19}.
\end{rem}

 \begin{rem}
 As the example provided in Section \ref{sec:example} shows, the notion of the sectorial extremal index is essentially weaker than the notion of the (global) extremal index. Indeed, the random field considered in this example has the {\em sectorial extremal index} $\theta=1$, while the (global) extremal index does not exist.
 \end{rem}

Within the theory of phantom distribution functions we have nice criteria for the existence of the extremal index and the sectorial extremal index.

\begin{theo}\label{theo11}
	Let $\{X_\bn : \bn\in\GZ^d\}$ be a stationary random field.
	Then $\{X_\bn\}$ has the extremal index $\theta\in (0,1]$ if, and only if,  there exist $\gamma_{or},\gamma_{in} \in (0,1)$ and a {\em strongly monotone} field of levels $\{v_{\bn}\,;\, \bn \in \GN^d\}$  such that
		\begin{equation}\label{eq:main11}
		 P( M_{\bn} \leq v_{\bn}) \to \gamma_{or}, \ \ F(v_{\bn})^{\bn^*} \to \gamma_{in},\ \text { as $\bn \to \pmb{\infty}$}, \quad \theta = \frac{\ln \gamma_{or}}{\ln \gamma_{in}},
		 \end{equation}
	and for every monotone curve  $\pmb{\psi}$ and every $T > 0$ {\bf Condition $\mathbf{B}_T^{\pmb{\psi}}(\{v_{\pmb{\psi}(n)}\})$} holds.	
\end{theo}
 
\begin{theo}\label{th:direct11}
	Let $\{X_\bn : \bn\in\GZ^d\}$ be a stationary random field.
	Then  $\{X_\bn\}$ has the  sectorial extremal index $\theta \in (0,1]$ if, and only if, 
	there exist $\gamma_{or}, \gamma_{in}  \in (0,1)$ and a {\em non-decreasing} sequence of levels $\{v_{\pmb{\Delta}(n)}\}$, $n \in \GN$, such that
		\begin{equation}\label{eq:mainbis11} 
		P( M_{\pmb{\Delta}(n)} \leq v_{\pmb{\Delta}(n)}) \to \gamma_{or},\ \ F(v_{\pmb{\Delta}(n)})^{n^d} \to \gamma_{in}, \ \  \text { as $n \to \infty$},\quad \theta = \frac{\ln \gamma_{or}}{\ln \gamma_{in}},
		\end{equation}
		and for every $T > 0$ Condition $\mathbf{B}_T^{\pmb{\Delta}}(\{v_{\pmb{\Delta}(n)}\})$ holds.	
\end{theo}

\section{Proofs}
\subsection{Proof of Proposition \ref{prop:zero}}

Clearly, if $G$ is a phantom distribution function for $\{X_{\bn}\}$, then it is a phantom distribution function for $\{X_{\bn}\}$ along every monotone curve. So assume the latter property and suppose that $G$ {\em does not} satisfy $(\ref{def_strong_pdf})$. It follows that there exists a number $\varepsilon > 0$, a  monotone sequence $\bm(n) \to \pmb{\infty}$ and a sequence $\{x_n\}$ such that 
\[\big| P\big( M_{\bm(n)} \leq x_n\big) - G(x_n)^{\bm(n)^*}\big| > \varepsilon, \ \ n \in \GN.\]
The point is that $\bm(n)$ need not satisfy (\ref{eq:crucial}) and so it is not a monotone curve according to our definition. 
But we can always find a monotone curve $\pmb{\psi}(n)$ such that 
$\bm(n) = \pmb{\psi}(m_n)$ for some increasing  sequence $\{m_n\}$. Indeed, let us begin with $\bm(1)$ and connect it with $\bm(2)$ by a sequence of points that in each step increases only by one in one coordinate. Then proceed the same way with points $\bm(2)$ and $\bm(3)$, etc.  The obtained map $\pmb{\psi}(\cdot) : \GN \to \GN^d$ satisfies (\ref{eq:crucial}). And $G$ cannot be a phantom distribution function for $\pmb{\psi}$.

\subsection{The mixing-like condition}\label{MIXING}

Let $\beta_T^{\pmb{\psi}}(n,k)$ for $n,k\in\GN $, $k\geq 2$, be defined~as
\begin{align*}
\beta^{\pmb{\psi}}_T(n,k):=\sup_{\mathbf{p}(1)+\ldots +\mathbf{p}(k)\leq T \pmb{\psi}(n)} 
\Bigg|P&(M_{\mathbf{p}(1)+\ldots +\mathbf{p}(k)}\leq v_{\pmb{\psi}(n)}) \\
&-\prod_{\bi \in\{1,\ldots,k\}^d} \!\!\!  P\left(M_{(p_1(i_1),\ldots,p_d(i_d))}  \leq  v_{\pmb{\psi}(n)}\right)\Bigg|,
\end{align*}
where $\mathbf{p}(1), \mathbf{p}(2), \ldots, \mathbf{p}(k)$ take values in $\GN_0^d$.
Then $\beta_T^{\pmb{\psi}}(n,2)=\beta_T^{\pmb{\psi}}(n)$ is the term appearing in the definition of Condition $\mathbf{B_T^{\pmb{\psi}}}(\{v_{\pmb{\psi}(n)}\})$. We are able to control the growth of $\beta_T^{\pmb{\psi}}(n,k)$.

\begin{lem}\label{L:ineq} The following inequality holds.
	\begin{equation}\label{eq:ineq}
	\beta^{\pmb{\psi}}_T(n,k)\leq k^{d}\beta^{\pmb{\psi}}_T(n), \quad k\geq 2.
	\end{equation}
\end{lem}
\begin{proof}
	Let us take $k\geq 3$ and $\mathbf{p}(1),\mathbf{p}(2),\ldots,\mathbf{p}(k)\in\GN_0^d$ satisfying the assumption $\mathbf{p}(1)+\mathbf{p}(2)+\ldots +\mathbf{p}(k)\leq T \pmb{\psi}(n)$. Define $\mathbf{q}(1):=\mathbf{p}(1)$, $\mathbf{q}(2):=\mathbf{p}(2)$, \ldots, $\mathbf{q}(k-2):=\mathbf{p}(k-2)$, $\mathbf{q}(k-1):=\mathbf{p}(k-1)+\mathbf{p}(k)$, so that 
	\[ \mathbf{q}(1)+\mathbf{2}(2)+\ldots +\mathbf{q}(k-1) =  \mathbf{p}(1)+\mathbf{p}(2)+\ldots +\mathbf{p}(k) \leq  T \pmb{\psi}(n).\]
	Then we obtain the following estimate.
	\begin{align*}
	&\left|P\left(M_{\mathbf{p}(1)+\ldots+\mathbf{p}(k)}\leq v_{\pmb{\psi}(n)}\right)-\prod_{\bj \in\{1,2,\ldots,k\}^d} \!\!\!  P\left(M_{(p_1(j_1),\ldots,p_d(j_d))}  \leq  v_{\pmb{\psi}(n)}\right)\right| \\
	& \leq  \left|P\left(M_{\mathbf{q}(1)+\mathbf{q}(2)+\ldots +\mathbf{q}(k-1)}\leq v_{\pmb{\psi}(n)}\right)-\prod_{\bi \in\{1,2,\ldots,k-1\}^d} \!\!\!  P\left(M_{(q_1(i_1),\ldots,q_d(i_d))}  \leq  v_{\pmb{\psi}(n)}\right)\right|\\
	& + \; \left|\prod_{\bi \in\{1,2,\ldots,k-1\}^d} \!\!\!  P\left(M_{(q_1(i_1),\ldots,q_d(i_d))}  \leq  v_{\pmb{\psi}(n)}\right)-\prod_{\bj \in\{1,2,\ldots,k\}^d} \!\!\!  P\left(M_{(p_1(j_1),\ldots,p_d(j_d))}  \leq  v_{\pmb{\psi}(n)}\right)\right| \\ 
	& \leq  \beta^{\pmb{\psi}}_T(n,k-1)+ \left| \Pi_1 - \Pi_2\right|.
	\end{align*}
	Let $\cD_k(r)$ consists of all  $\bi = (i_1,i_2,\ldots, i_d) \in \{1,2,\ldots, k-1\}^d$ such that the number of $s$ with the property that $i_s = k-1$ equals $r$.
	Next, for  $\bi \in  \cD_k(r)$ define  $\cE_k(r,\bi)$ 
	as the set of $\bj = (j_1,j_2,\ldots, j_d) \in \{1,2,\ldots, k\}^d$ 
	such that $j_s = i_s$, if $i_s \neq k-1$ and $j_s \in \{k-1,k\}$, if $i_s = k-1$. 
	Let us observe that for $\bi \in  \cD_k(0)$  we have
	 $\cE_k(0,\bi) = \{\bi\}$ and that for each 
	 $\bi= (i_1,i_2,\ldots, i_d)\in  \cD_k(r)$
	\[ \left| P\left(M_{(q_1(i_1),\ldots,q_d(i_d))}  \leq  
	v_{\pmb{\psi}(n)}\right) -
	\prod_{\bj \in \cE_k(r,\bi)} P\left(M_{(p_1(j_1),\ldots,p_d(j_d))}  \leq  v_{\pmb{\psi}(n)}\right) \right| \leq \beta^{\pmb{\psi}}_T(n). \]
	Taking into account these relations and using the obvious expansions:
	\begin{align*}
	\Pi_1 &= \prod_{r=0}^{d} \prod_{\bi\in \cD_k(r)}  P\left(M_{(q_1(i_1),\ldots,q_d(i_d))}  \leq  v_{\pmb{\psi}(n)}\right), \\
	\Pi_2 &= \prod_{r=0}^{d} \prod_{\bi\in \cD_k(r)} \prod_{\bj \in \cE_k(r,\bi)} P\left(M_{(p_1(j_1),\ldots,p_d(j_d))}  \leq  v_{\pmb{\psi}(n)}\right),
	\end{align*}
	we obtain that 
	\begin{align*}
	 \big| \Pi_1 - \Pi_2 \big| &\leq \beta^{\pmb{\psi}}_T(n) \sum_{r=1}^d \#\cD_k(r) = \big((k-1)^d - (k-2)^d\big)\beta^{\pmb{\psi}}_T(n)\\ 
	 &= \Big(\sum_{r=0}^{d-1} (k-1)^{d-1 - r} (k-2)^r\Big) \beta^{\pmb{\psi}}_T(n) \leq d  (k-1)^{d-1} \beta^{\pmb{\psi}}_T(n).
	 \end{align*}
	It follows that for $k \geq 3$
	\[ \beta^{\pmb{\psi}}_T(n,k)\leq \beta^{\pmb{\psi}}_T(n,k-1) + d (k-1)^{d-1}\beta^{\pmb{\psi}}_T(n). \]
	Iterating the above relation we get (\ref{eq:ineq}).
	%\[ \beta^{\pmb{\psi}}_T(n,k) \leq \big(\sum_{l=1}^{k-1} l^d\big) %\beta^{\pmb{\psi}}_T(n)  \leq \frac{1}{d} k^{d+1} %\beta^{\pmb{\psi}}_T(n,k).\]
\end{proof}

\begin{lem}\label{lem:fixed}
Let $\bN(n) = \big(N_1(n), N_2(n), \ldots, N_d(n)\big) \in \GN^d$,  $\bN(n)\to \pmb{\infty}$. Suppose that $q_1, q_2, \ldots, q_d \in \GN$
are such that for some $T_0 > 0$, 
$\big(q_1 N_1(n), q_2 N_2(n), \ldots, q_d N_d(n)\big) \leq T_0 \pmb{\psi}(n)$, $n\in\GN$. 	If Condition
 $\mathbf{B}_{T_0}^{\pmb{\psi}}(\{v_{\pmb{\psi}(n)}\})$ holds, then we have, as $n\to\infty$,
	\begin{equation}\label{weak_dep1}
	\begin{split}
P\big(&M_{(q_1 N_1(n),q_2 N_2(n),\ldots, q_d N_d(n))}\leq v_{\pmb{\psi}(n)}\big) \\
&\qquad\qquad 
= P\big(M_{(N_1(n),N_2(n),\ldots,N_d(n))}\leq v_{\pmb{\psi}(n)}\big)^{q_1 q_2 \ldots q_d}+o(1).
\end{split}
\end{equation}
\end{lem}
\begin{proof} Fix $n\in\GN$.
We can represent $\big(q_1 N_1(n), q_2 N_2(n), \ldots, q_d N_d(n)\big)$ as the sum of $s = q_1 + q_2 + \ldots +q_d$ specific components, namely $q_1$ components $\big( N_1(n),0,\ldots,0)$, $q_2$ components $\big(0, N_2(n),0,\ldots,0)$, etc. Keeping the order, let us denote these components by $\bp(1), \bp(2), \ldots, \bp(s)$. By Lemma \ref{L:ineq}
\[P\big(M_{(q_1 N_1(n),q_2 N_2(n),\ldots, q_d N_d(n))}\leq v_{\pmb{\psi}(n)}\big)
-\prod_{\bi \in\{1,\ldots,s\}^d} \!\!\!  P\left(M_{(p_1(i_1),\ldots,p_d(i_d))}  \leq  v_{\pmb{\psi}(n)}\right) \to 0.\]
It remains to identify 
\[ \prod_{\bi \in\{1,\ldots,s\}^d} \!\!\!  P\left(M_{(p_1(i_1),\ldots,p_d(i_d))}  \leq  v_{\pmb{\psi}(n)}\right) \] with
\[ P\big(M_{(N_1(n),N_2(n),\ldots,N_d(n))}\leq v_{\pmb{\psi}(n)}\big)^{q_1 q_2 \ldots q_d}.\] 
Consider a typical term $\bP_{\bi} = P\left(M_{(p_1(i_1),\ldots,p_d(i_d))}  \leq  v_{\pmb{\psi}(n)}\right)$, $\bi \in \{1,2,\ldots, s\}^d$. If some coordinate $p_j(i_k)$ is $0$ then $\bP_{\bi} = 1$, for we have $\max \emptyset = -\infty$ by the well-known convention. If all coordinates are non-zero, then 
$p_1(i_1) = N_1(n), p_2(i_2) = N_2(n), \ldots, p_d(i_d) = N_d(n)$, $\bP_{\bi} =  P\left(M_{(N_1(n),\ldots,N_d(n))}  \leq  v_{\pmb{\psi}(n)}\right)$ and this can be achieved in $q_1 q_2 \ldots q_d$ ways.
\end{proof}

\begin{corol}\label{cor:fixed} In assumptions of Lemma \ref{lem:fixed}, if Condition
	$\mathbf{B}_{T}^{\pmb{\psi}}(\{v_{\pmb{\psi}(n)}\})$ is satisfied for every $T > 0$, then (\ref{weak_dep1}) holds for any $q_1, q_2, \ldots, q_d \in \GN$. 
\end{corol}

\begin{corol}\label{cor:kaen}
	Suppose that  $\{\bN(n)\} \subset \GN^d$, $\bN(n)\to \pmb{\infty}$, $\{\bk(n)\} \subset \GN^d$ and for some $T_0 > 0$
 \[ \big(k_1(n)\bN_1(n), k_2(n)\bN_2(n), \ldots, k_d(n)\bN_d(n)\big) \leq T_0 \pmb{\psi}(n), \ n\in\GN.\]
  	If Condition
$\mathbf{B}_{T_0}^{\pmb{\psi}}(\{v_{\pmb{\psi}(n)}\})$ holds and $\big(k_1(n) + \ldots + k_d(n)\big)^d\beta^{\pmb{\psi}}_{T_0}(n)\to 0$, as $n\to\infty$, then 
	\begin{equation}\label{weak_dep1a}
P\big(M_{\big(k_1(n)\bN_1(n), k_2(n)\bN_2(n), \ldots, k_d(n)\bN_d(n)\big)} \leq v_{\pmb{\psi}(n)}\big)  
= P\big(M_{\bN(n)}\leq v_{\pmb{\psi}(n)}\big)^{\bk(n)^*}+o(1).
\end{equation}
\end{corol}
\begin{proof} Proof follows by a careful inspection of the proof of  Lemma \ref{lem:fixed}.
\end{proof}

In the sequel $\lfloor x\rfloor$ will denote the integer part of $x\in\GR^1$. Similarly, if $\bx = (x_1,x_2,\ldots, x_d) \in \GR^d$, then $\lfloor \bx \rfloor$ is the vector if integer parts of coordinates:
\[  \lfloor \bx \rfloor = \big(\lfloor x_1 \rfloor, \lfloor x_2 \rfloor, \ldots,\lfloor x_d \rfloor\big)\]

The next fact is of independent interest and therefore for the future purposes we state it as a theorem.

\begin{theo}\label{remark_BT}  Let $\{\bN(n)\}\subset \GN^d$, $\bN(n)\to\pmb{\infty}$ and satisfies $\bN(n)\leq T_0 \pmb{\psi}(n)$, $n \in \GN$, for some $T_0 > 0$. Let Condition $\mathbf{B}_{T_0(1 + \varepsilon)}^{\pmb{\psi}}(\{v_{\pmb{\psi}(n)}\})$ holds, for some $\varepsilon > 0$.
	
Suppose that $k_n \to \infty$ in such a way that as $n\to\infty$ both $k_n^d\beta^{\pmb{\psi}}_{T_0}(n)\to 0$
 and $k_n=o(N_i(n))$, $i = 1,2,\ldots, d$.
 
  Then, as $n\to\infty$,
		\begin{align}\label{weak_dep}
		P\big(&M_{\bN(n)}\leq v_{\pmb{\psi}(n)}\big) = P\big(M_{(\lfloor N_1(n)/k_n\rfloor,\lfloor N_2(n)/k_n\rfloor,\ldots,\lfloor N_d(n)/k_n\rfloor)}\leq v_{\pmb{\psi}(n)}\big)^{k_n^d}+o(1) \\
	\label{weak_dep_2}
		&\qquad\quad = \exp\big(-k_n^dP\big(M_{(\lfloor N_1(n)/k_n\rfloor,\lfloor N_2(n)/k_n\rfloor,\ldots,\lfloor N_d(n)/k_n\rfloor)}>
		 v_{\pmb{\psi}(n)}\big)\big)+o(1),
		\end{align}
\end{theo}
\begin{proof}
 From Corollary \ref{cor:kaen} we obtain that 
	\begin{align*}
		P\left(M_{\bN(n)}\leq v_{\pmb{\psi}(n)}\right) 
		&\leq P\left(M_{(k_n\lfloor N_1(n)/k_n\rfloor,k_n\lfloor N_2(n)/k_n\rfloor,\ldots,k_n\lfloor N_d(n)/k_n\rfloor)}\leq v_{\pmb{\psi}(n)}\right)\\
		&= P\left(M_{(\lfloor N_1(n)/k_n\rfloor,\lfloor N_2(n)/k_n\rfloor,\ldots,\lfloor N_d(n)/k_n\rfloor)}\leq v_{\pmb{\psi}(n)}\right)^{k_n^d} + o(1) =: V_n.
	\end{align*}
To get the other bound, for each $n\in \GN$ find numbers $l_{n,1}, l_{n,2}, \ldots, l_{n,d}$ in $\GN$ such that 
\[
 (k_n+l_{n,i}-1)\lfloor N_i(n)/k_n \rfloor \leq N_i(n) < (k_n+l_{n,i})\lfloor N_i(n)/k_n \rfloor,\ i = 1,2,\ldots, d. 
 \]
 In other words,
 \[l_{n,i}=\left\lfloor  \frac{N_i(n) - k_n \lfloor N_i(n)/k_n \rfloor}{\lfloor N_i(n)/k_n\rfloor} \right\rfloor +1,\]
 what implies 
 \begin{equation}\label{eq:elen}
 l_{n,i}=o(k_n), \ i=1,2,\ldots,d.
 \end{equation}
  This in turn implies that for large $n$
 \[ (k_n+l_{n,i})\lfloor N_i(n)/k_n \rfloor 
 \leq T_0 (1 + \varepsilon_0)\pmb{\psi}(n),\ i = 1,2,\ldots, d.\]
 Therefore we can again apply Corollary \ref{cor:kaen}. 	
\begin{align*}
		P\big(&M_{\bN(n)}\leq v_{\pmb{\psi}(n)}\big)\\
		&\geq   P\big(M_{((k_n + l_{n,1})\lfloor N_1(n)/k_n\rfloor, (k_n+ l_{n,2})\lfloor N_2(n)/k_n\rfloor,\ldots,(k_n + l_{n,d})\lfloor N_d(n)/k_n\rfloor)}\leq v_{\pmb{\psi}(n)} \big) \\
			&=   P\big(M_{(\lfloor N_1(n)/k_n\rfloor,\lfloor N_2(n)/k_n\rfloor,\ldots,\lfloor N_d(n)/k_n\rfloor)}\leq v_{\pmb{\psi}(n)}\big)^{\prod_{i=1}^d(k_n+l_{n,i})} + o(1) =: U_n.
	\end{align*}
From (\ref{eq:elen}) we get $U_n-V_n=o(1)$ and so (\ref{weak_dep}) holds.

Relation (\ref{weak_dep_2}) is equivalent to (\ref{weak_dep}), since $(a_m)^m-\exp(-m(1-a_m))\to 0$, as $m\to\infty$, for arbitrary $\{a_m\}\subset [0,1]$.
\end{proof}

\begin{prop}\label{pro:pegen}  
	Let $\{\bR(n)\}\subset \GR^d_+$, $\bR(n)\to\pmb{\infty}$ and $q_1, q_2,\ldots, q_d \in \GN$.
	Suppose that  for some $T_0 > 0$
	\[ \big(q_1 R_1(n), q_1 R_2(n), \ldots, q_d R_d(n)\big) \leq T_0 \pmb{\psi}(n), \ \ n \in \GN.\] 
 If for some $\varepsilon > 0$ Condition $\mathbf{B}_{T_0(1 + \varepsilon)}^{\pmb{\psi}}(\{v_{\pmb{\psi}(n)}\})$ holds, then,
	as $n\to\infty$,
		\begin{equation}\label{weak_dep1c}
	\begin{split}
	P\big(&M_{(\lfloor q_1 R_1(n)\rfloor,\lfloor q_2 R_2(n)\rfloor,\ldots, \lfloor q_d R_d(n)\rfloor)}\leq v_{\pmb{\psi}(n)}\big) \\
	&\qquad\qquad 
	= P\big(M_{(\lfloor R_1(n)\rfloor ,\lfloor R_2(n)\rfloor, \ldots,\lfloor R_d(n)\rfloor)}\leq v_{\pmb{\psi}(n)}\big)^{q_1 q_2 \ldots q_d}+o(1).
	\end{split}
	\end{equation}
\end{prop}

\begin{proof}	Let us notice first that 
	\begin{align*} 
		P\big(&M_{(\lfloor q_1 R_1(n)\rfloor,\lfloor q_2 R_2(n)\rfloor,\ldots, \lfloor q_d R_d(n)\rfloor)}\leq v_{\pmb{\psi}(n)}\big) \\
	&\qquad\qquad 
	\leq  P\big(M_{(q_1\lfloor R_1(n)\rfloor , q_2\lfloor R_2(n)\rfloor, \ldots,q_d \lfloor R_d(n)\rfloor)}
	\leq v_{\pmb{\psi}(n)}\big) \\
	&\qquad\qquad 
	=  P\big(M_{(\lfloor R_1(n)\rfloor, \lfloor R_2(n)\rfloor, \ldots,\lfloor R_d(n)\rfloor)}
	\leq v_{\pmb{\psi}(n)}\big) ^{q_1 q_2 \ldots q_d}+o(1),
	\end{align*}
where the last equality holds by Lemma \ref{lem:fixed}. Therefore it is enough to find expressions $U_n$ and $V_n$ such that $V_n - U_n = o(1)$ and $U_n \leq  	P\big(M_{(\lfloor q_1 R_1(n)\rfloor,\lfloor q_2 R_2(n)\rfloor,\ldots, \lfloor q_d R_d(n)\rfloor)}\leq v_{\pmb{\psi}(n)}\big)$, while $V_n \geq 
	P\big(M_{(q_1\lfloor R_1(n)\rfloor , q_2\lfloor R_2(n)\rfloor, \ldots,q_d \lfloor R_d(n)\rfloor)}
	\leq v_{\pmb{\psi}(n)}\big)$. 

%We shall apply Proposition \ref{remark_BT}. 
Let $r_n \to \infty$ in such a way that $r_n^d  \beta^{\pmb{\psi}}_{T_0(1+\varepsilon)}(n)\to 0$ and $r_n = o(R_i(n))$, $i=1,2,\ldots, d$. Then for $n$ large enough we have $q_i \lfloor R_i(n) \rfloor \geq (r_n - 1) \lfloor q_i R_i(n)/r_n \rfloor$, $i =1,2,\ldots, d$, and therefore by Corollary \ref{cor:kaen}
\begin{align*}
 P\big(&M_{(q_1\lfloor R_1(n)\rfloor , q_2\lfloor R_2(n)\rfloor, \ldots,q_d \lfloor R_d(n)\rfloor)} 
\leq v_{\pmb{\psi}(n)}\big)\\ 
&\leq P\big(M_{\big((r_n-1)\lfloor q_1 R_1(n)/r_n\rfloor , (r_n-1)\lfloor q_2 R_2(n)/r_n\rfloor, \ldots, (r_n-1)\lfloor q_d R_d(n)/r_n\rfloor \big)}
\leq v_{\pmb{\psi}(n)}\big) \\
&= P\big(M_{\big((\lfloor q_1 R_1(n)/r_n\rfloor, \lfloor q_2 R_2(n)/r_n\rfloor, \ldots, \lfloor q_d R_d(n)/r_n\rfloor \big)}
\leq v_{\pmb{\psi}(n)}\big)^{(r_n - 1)^d} + o(1) := V_n.
\end{align*}

In order to find $U_n$ we shall proceed like in the proof of Theorem \ref{remark_BT}. Let $s_{n,i} \in \GN$, $i=1,2,\ldots, d$, be such that	 
\[
(r_n+s_{n,i}-1)\lfloor q_i R_i(n)/r_n \rfloor \leq \lfloor q_i R_i(n) \rfloor < (r_n+s_{n,i})\lfloor q_i R_i(n)/r_n \rfloor,\ i = 1,2,\ldots, d,
\]
or, equivalently,
\[ s_{n,i} = \Big\lfloor \frac{\lfloor q_i R_i(n)\rfloor - r_n \lfloor q_i R_i(n)/r_n\rfloor}{\lfloor q_i R_i(n)/r_n\rfloor} \Big\rfloor + 1.\]
Applying Corollary \ref{cor:kaen}, we get 	
\begin{align*}
	P\big(&M_{(\lfloor q_1 R_1(n)\rfloor,\lfloor q_2 R_2(n)\rfloor,\ldots, \lfloor q_d R_d(n)\rfloor)}\leq v_{\pmb{\psi}(n)}\big)\\
&\geq   P\big(M_{((r_n + s_{n,1})\lfloor q_1 R_1(n)/r_n\rfloor, (r_n+ s_{n,2})\lfloor q_2 R_2(n)/r_n\rfloor,\ldots,(r_n + s_{n,d})\lfloor q_d R_d(n)/r_n\rfloor)}\leq v_{\pmb{\psi}(n)} \big) \\
&=   P\big(M_{(\lfloor q_1 R_1(n)/r_n\rfloor,\lfloor q_2 R_2(n)/r_n\rfloor,\ldots,\lfloor q_d R_d(n)/r_n\rfloor)}\leq v_{\pmb{\psi}(n)}\big)^{\prod_{i=1}^d(r_n+s_{n,i})} + o(1) =: U_n.
\end{align*}
Since $s_{n,i} = o(r_n)$, $i=1,2,\ldots,d$, we get  
$U_n-V_n=o(1)$ and so (\ref{weak_dep1c}) holds.
\end{proof}

\subsection{Fields of monotone levels}
In this section we shall examine  previous results in conjunction with properties of the sequence of levels $\{v_{\pmb{\psi}(n)}\}$.

\begin{prop}\label{Fstar}
	If Condition  $\mathbf{B}_1^{\pmb{\psi}}(\{v_{\pmb{\psi}(n)}\})$ holds for a monotone sequence of levels $\{v_{\pmb{\psi}(n)}\}$ that satisfies (\ref{eq:mainbis}),
	 then
	\begin{equation}\label{eq:rightend}
	v_{\pmb{\psi}(n)} \nearrow F_*,
	\end{equation}
	where $F_* = \sup \{ x \,:\, F(x) < 1 \}$. 
\end{prop}

\begin{proof}
	If $v_{\pmb{\psi}(n_0)} \geq F_*$ for some $n_0$, then $ P(M_{\pmb{\psi}(n)}\leq v_{\pmb{\psi}(n)}) = 1$ for all $n \geq n_0$ and (\ref{eq:mainbis}) cannot hold. So assume that for some $\eta > 0$  $v_{\pmb{\psi}(n)} \leq (1 -\eta) F_*$, $n \in \GN$. Then for some $a > 0$ we have $ P( X_{\pmb{1}} \leq  v_{\pmb{\psi}(n)}) \leq 1 - a$, $n\in \GN$.
	
	Let $k_n \to \infty$ in such a way that $k_n^{d}  \beta^{\pmb{\psi}}_1(n) \to 0$. Then  by (\ref{weak_dep})
	\begin{align*}
	P(M_{\pmb{\psi}(n)}\leq v_{\pmb{\psi}(n)}) &\leq P(M_{k_n \lfloor \pmb{\psi}(n)/k_n \rfloor}\leq v_{\pmb{\psi}(n)})
	= P(M_{\lfloor \pmb{\psi}(n)/k_n \rfloor}\leq v_{\pmb{\psi}(n)})^{k_n^d} + o(1)\\ &\leq P( X_{\mathbf{1}} \leq v_{\pmb{\psi}(n)})^{k_n^d} +o(1) \leq (1-a)^{k_n^d} + o(1)\to 0.
	\end{align*}
	This again contradicts  (\ref{eq:mainbis}) and so 
	$v_{\pmb{\psi}(n)} \nearrow F_*$.
\end{proof}

\begin{prop}\label{prop:unif}
	Suppose (\ref{eq:mainbis}) holds for some monotone sequence of levels $\{v_{\pmb{\psi}(n)}\}$ and some $\gamma \in (0,1)$ and Condition $\mathbf{B}_T^{\pmb{\psi}}(\{v_{\pmb{\psi}(n)}\})$ holds for every $T>0$. 
	\begin{description}
		\item{\bf (i)} For every $d$-tuple $\bt=(t_1,t_2,\ldots,t_d) \in (0,\infty)^d$,
		\begin{equation}\label{eq:point}
		P\left(M_{(\lfloor t_1 \psi_1(n)\rfloor,\lfloor t_2 \psi_2(n)\rfloor,\ldots,\lfloor t_d \psi_d(n)\rfloor)}\leq v_{\pmb{\psi}(n)}\right) \xrightarrow[n\to\infty]{} \gamma^{\;t_1t_2\cdots t_d}.
		\end{equation}
		\item{\bf (ii)} If a set $A\subset[0,\infty)^d$ does not contain any sequence $\{\mathbf{t}(n)\}$ with the property that $t_{i_1}(n)\to\infty$ and $t_{i_2}(n)\to 0$ for some $i_1 \neq i_2\in\{1,2,\ldots,d\}$, then
		\begin{equation*}
		\sup_{\mathbf{t}\in A}\left|P\left(M_{(\lfloor t_1 \psi_1(n)\rfloor ,\lfloor t_2 \psi_2(n)\rfloor,\ldots,\lfloor t_d \psi_d(n)\rfloor )}\leq v_{\pmb{\psi}(n)}\right)- \gamma ^{\; t_1t_2\cdots t_d}\right|\xrightarrow[n\to\infty]{} 0.
		\end{equation*}
	\end{description}
\end{prop}
\begin{proof}
	First consider $t_1 = 1/q_1, t_2 = 1/q_2, \ldots, t_d = 1/q_d$, where $q_1, q_2, \ldots, q_d \in \GN$. Set $R_i(n) =
	\pmb{\psi}_i(n)/q_i$. By Proposition \ref{pro:pegen},
	\[ \gamma \longleftarrow P(M_{\pmb{\psi}(n)}\leq v_{\pmb{\psi}(n)}) = 
	P(M_{(\lfloor \psi_1(n)/q_1 \rfloor,  \lfloor \psi_2(n)/q_2 \rfloor, \ldots, \lfloor \psi_d(n)/q_d \rfloor)}\leq v_{\pmb{\psi}(n)})^{q_1 q_2\cdots q_d} + o(1),\]
	hence  
	\[ P(M_{(\lfloor \psi_1(n)/q_1 \rfloor,  \lfloor \psi_2(n)/q_2 \rfloor, \ldots, \lfloor \psi_d(n)/q_d \rfloor}\leq v_{\pmb{\psi}(n)}) \longrightarrow \gamma^{1/(q_1 q_2 \cdots q_d)} = \gamma^{\;t_1t_2\cdots t_d}.\]
	By another application of Proposition \ref{pro:pegen} we have for $p_1, p_2, \ldots, p_d \in \GN$,
	\begin{align*}
	&P(M_{(\lfloor p_1 \psi_1(n)/q_1 \rfloor,  \lfloor p_2 \psi_2(n)/q_2 \rfloor, \ldots, \lfloor p_d \psi_d(n)/q_d \rfloor)}\leq v_{\pmb{\psi}(n)})\\
	&\qquad\qquad\qquad= P(M_{(\lfloor \psi_1(n)/q_1 \rfloor,  \lfloor \psi_2(n)/q_2 \rfloor\ldots, \lfloor \psi_d(n)/q_d \rfloor)}\leq v_{\pmb{\psi}(n)})^{p_1 p_2 \cdots p_d} + o(1) \\
	&\qquad\qquad\qquad= \gamma^{\frac{p_1 p_2 \cdots p_d}{q_1 q_2 \cdots q_d}} + o(1) = \gamma^{t_1 t_2 \cdots t_d}+ o(1),
	\end{align*} 
	if $t_1 = p_1/q_1, t_2 = p_2/q_2,\ldots, t_d= p_d/q_d$.
	
	We have proved (\ref{eq:point}) over the countable dense set $\GQ_+^d$. The pointwise convergence over $\GR_+^d$ follows then by the monotonicity of maps 
	\[\mathbf{s}\mapsto P\left(M_{(\lfloor s_1 \psi_1(n)\rfloor,\lfloor s_2 \psi_2(n)\rfloor,\ldots,\lfloor s_d \psi_d(n)\rfloor)}\leq v_{\pmb{\psi}(n)}\right)\] and the continuity of the limiting map $\mathbf{s}\mapsto \gamma^{\;s_1s_2\cdots s_d}$.
	
	Part (ii) of Proposition \ref{prop:unif} is, in fact,  a general statement on convergence of monotone functions to a continuous function on $[0,\infty)^d$.  
	For the sake of notational simplicity we shall restrict our attention to the case $d=2$. The general case can be proved analogously.
	
	Let $A\subset[0,\infty)^2$ be a set fulfilling the assumptions of part (ii) of Proposition \ref{prop:unif}.  Let $\{\mathbf{t}(n)\} \subset A$ be a sequence converging to some $\bt\in[0,\infty]^2$. We have to prove that 
	\begin{equation}\label{NAT2}
	P\left(M_{(\lfloor t_1(n)\psi_1(n)\rfloor ,\lfloor t_2(n)\psi_2(n)\rfloor )}\leq v_{\pmb{\psi}(n)}\right) - \gamma ^{\;t_1(n)t_2(n)} \xrightarrow[n\to\infty]{} 0.
	\end{equation}
	We shall consider the following three situations: (a) $\mathbf{t}\in(0,\infty)^2$; (b)~$\max\{t_1,t_2\}<\infty$ and $\min\{t_1,t_2\}=0$; (c) $\max\{t_1,t_2\}=\infty$ and $\min\{t_1,t_2\}>0$. The case (d) $\max\{t_1,t_2\}=\infty$ and $\min\{t_1,t_2\}=0$ is excluded by the assumptions on the set $A$.
	
	Suppose that $\mathbf{t}\in(0,\infty)^2$. Then $(t_1-\varepsilon,t_2-\varepsilon)\leq (t_1(n),t_2(n)) \leq (t_1+\varepsilon,t_2+\varepsilon)$ for sufficiently large $n\in\GN $ and every $\varepsilon > 0$. By the monotonicity and part (i) we get for small $\varepsilon$
	\begin{align*}
	\gamma^{\;(t_1+\varepsilon)(t_2+\varepsilon)} \longleftarrow&P\left(M_{(\lfloor (t_1+\varepsilon)\psi_1(n)\rfloor ,\lfloor (t_2+\varepsilon)\psi_2(n)\rfloor )}\leq v_{\pmb{\psi}(n)}\right)\\ 
	&\qquad\leq P\left(M_{(\lfloor t_1(n)\psi_1(n)\rfloor ,\lfloor t_2(n)\psi_2(n)\rfloor )}\leq v_{\pmb{\psi}(n)}\right) \\
	&\qquad\qquad\leq P\left(M_{(\lfloor (t_1-\varepsilon)\psi_1(n)\rfloor ,\lfloor (t_2-\varepsilon)\psi_2(n)\rfloor )}\leq v_{\pmb{\psi}(n)}\right) \longrightarrow \gamma^{\;(t_1-\varepsilon)(t_2-\varepsilon)}.
	\end{align*}
	Hence 
	\[ \lim_{n\to\infty} P\left(M_{(\lfloor t_1(n)\psi_1(n)\rfloor ,\lfloor t_2(n)\psi_2(n)\rfloor )}\leq v_{\pmb{\psi}(n)}\right) =  \gamma^{\; t_1t_2} = \lim_{n\to\infty} \gamma^{\; t_1(n)t_2(n)},\]
	and condition (\ref{NAT2})  is satisfied in case (a).
	
	Now consider $\mathbf{t}=(t_1,0)$ with $t_1\in[0,\infty)$. Then $\gamma^{\;t_1(n)t_2(n)}\to 1$. Similarly, for every $\varepsilon>0$ we have by part (i)
	\begin{align*}
	 1 \geq &P\left(M_{(\lfloor t_1(n) \psi_1(n)\rfloor,\lfloor t_2(n) \psi_2(n)\rfloor)}
	  \leq v_{\pmb{\psi}(n)}\right)\\
	  &\qquad\qquad\geq  P\left(M_{(\lfloor(t_1+\varepsilon)\psi_1(n)\rfloor,\lfloor\varepsilon \psi_2(n)\rfloor)}\leq v_{\pmb{\psi}(n)}\right) \to \gamma^{(t_1+\varepsilon)\varepsilon}.
	 \end{align*}	Passing with $\varepsilon \to 0$ gives us (\ref{NAT2}) in case (b).
	
	Next assume that $\mathbf{t}=(\infty,t_2)$ for some $t_2\in(0,\infty]$. Then $\gamma^{\;t_1(n)t_2(n)}\to 0$. Moreover, for all $R>0$, $\varepsilon>0$ and sufficiently large $n\in\GN$ we have
	\begin{align*}
	0\leq &P\left(M_{(\lfloor t_1(n)\psi_1(n)\rfloor ,\lfloor t_2(n)\psi_2(n)\rfloor )} \leq v_{\pmb{\psi}(n)} \right) \\
	&\qquad\qquad\leq P\left(M_{(\lfloor R\psi_1(n)\rfloor ,\lfloor (t_2-\varepsilon)\psi_2(n)\rfloor )} \leq v_{\pmb{\psi}(n)} \right) \longrightarrow \gamma^{(t_2-\varepsilon)R}.
	\end{align*}
	Passing with $R \to \infty$ gives (\ref{NAT2}) in case (c) and completes the proof of part (ii) of the proposition.
\end{proof} 

\subsection{Proof of Theorem \ref{theo1}}
\subsubsection{Necessity}\label{Sec221}

Suppose that $G$ is a continuous distribution function. Take $\gamma \in (0,1)$ and for $\bn \in \GN^d$ define
\[ v_{\bn} = \inf \{ x\,:\, G(x)^{\bn^*} = \gamma\}.\]
Then the field of levels $\{v_{\bn}\}$ is strongly monotone.

If $G$ is a phantom distribution function for $\{X_{\bn}\}$, then 
\[ P\left( M_{\bn} \leq v_{\bn}\right) = G(v_{\bn})^{\bn^*} + o(1) = \gamma + o(1),\]
hence condition (i) of the theorem is satisfied.

Next let $\pmb{\psi}$ be a monotone curve and let $T > 0$. 
We want to verify Condition $\mathbf{B}_T^{\pmb{\psi}}(\{v_{\pmb{\psi}(n)}\})$.
Assume that $\mathbf{p}(n)\to\mathbf{\infty}$ and $\mathbf{q}(n) \to \mathbf{\infty}$ satisfy additionally 
\[  \mathbf{p}(n)+ \mathbf{q}(n) \leq T \pmb{\psi}(n), \ n\in \GN.\]
Passing to a subsequence, if necessary, we can assume that 
\[ \frac{p_i(n)}{\psi_i(n)} \to s_i \in [0,T],\ \ \frac{q_i(n)}{\psi_i(n)} \to t_i \in [0,T],\ \ i=1,2,\ldots, d.\]
We have 
\[ P\big( M_{\bp(n) + \bq(n)} \leq v_{\pmb{\psi}(n)}\big) = 
G(v_{\pmb{\psi}(n)})^{(\bp(n) + \bq(n))^*} = G(v_{\pmb{\psi}(n)})^{\pmb{\psi}(n)^* \frac{(\bp(n) + \bq(n))^*}{\pmb{\psi}(n)^*}} \longrightarrow \gamma^{\;\prod_{i=1}^d (s_i + t_i)}.\]
Consider the following expansion.
\[ \prod_{i=1}^d (s_i + t_i) = \sum_{I_0 \subset \{1,2,\ldots ,d\}} \prod_{i\in I_0} s_i \times \prod_{j\not\in I_0} t_j = \sum_{I_0 \subset \{1,2,\ldots ,d\}} \Pi_{I_0}.\]
It is clear that each term $\gamma^{\Pi_{I_0}}$ is a common limit for both $G(v_{\pmb{\psi}(n)})^{\br(n)^*}$ and $P(M_{\br(n)} \leq v_{\pmb{\psi}(n)})$, where
\[ r_i(n) = \begin{cases} p_i(n), &\text{ if  }\ i \in I_0;\\
q_i(n), &\text{ if  }\ i \not\in I_0.\end{cases} \]
We have proved that the difference between the two expressions appearing in Condition $\mathbf{B}_T^{\pmb{\psi}}(\{v_{\pmb{\psi}(n)}\})$ tends to zero.

The same is also true  if some coordinate of $\bp(n)$ or  $\bq(n)$ remains bounded along a subsequence, since then the corresponding terms in the expansion converge to 1. Indeed, suppose that e.g. $p_1(n) \leq K$, $n\in\GN$. Then for large $n$
\begin{align*}
\lim_{n\to\infty}P(M_{(p_1(n), \ldots, p_d(n))}\leq v_{\pmb{\psi}(n)}) &\geq \lim_{n\to\infty}P(M_{(\lfloor\varepsilon\psi_1(n)\rfloor,\lfloor T\psi_2(n)\rfloor,\ldots,  \lfloor T\psi_d(n)\rfloor)}\leq v_{\pmb{\psi}(n)}) \\
&=\lim_{n\to\infty} G(v_{\pmb{\psi}(n)})^{\varepsilon T^{d-1} {\pmb{\psi}(n)}^*}=\gamma^{\; \varepsilon T^{d-1}} \nearrow 1,\quad \text{ as } \varepsilon \to 0.
\end{align*}

\subsubsection{Sufficiency} 

Let $\{v_{\bn}\}$ be a strongly monotone field of levels such that 
$ P\left( M_{\bn} \leq v_{\bn}\right) \longrightarrow \gamma$,
for some $\gamma \in (0,1)$. 

We shall show that along every monotone curve $\pmb{\psi}(n)$ there exists a continuous phantom distribution function $G_{\pmb{\psi}}$ and that all these functions are strictly  tail-equivalent in the sense of \cite{DJL15}, i.e. 
if  $G_{\pmb{\psi}'}$ and  $G_{\pmb{\psi}^{\prime\prime}}$ are phantom distribution functions along monotone curves $\pmb{\psi}'$ and $\pmb{\psi}^{\prime\prime}$, respectively, then 
\begin{equation}\label{eqn:tailequiv}
(G_{\pmb{\psi}'})_* = (G_{\pmb{\psi}^{\prime\prime}})_* \quad\text{and}\quad  \frac{1 - G_{\pmb{\psi}'}(x)}{1 - G_{\pmb{\psi}^{\prime\prime}}(x)} \to 1, \text{ as  $x \to (G_{\pmb{\psi}'})_*- $.} 
\end{equation}
Applying \cite[Proposition 1, p. 700]{DJL15} one gets that 
\begin{equation}\label{eq:equiv}
\sup_{x\in\GR} \left| G_{\pmb{\psi}'}(x)^n - G_{\pmb{\psi}^{\prime\prime}}(x)^n \right| \longrightarrow 0.
\end{equation}
If (\ref{eq:equiv}) holds for all pairs $\pmb{\psi}'$ and $\pmb{\psi}^{\prime\prime}$, then it is enough to set $G = G_{\pmb{\Delta}}$, where $\pmb{\Delta}(n) = (n,n,\ldots, n)$.

So let us take any monotone curve $\pmb{\psi}(n)$ and assume that Condition $\mathbf{B}_T^{\pmb{\psi}}(\{v_{\pmb{\psi}(n)}\})$ holds for every $T > 0$.

We define $G_{\pmb{\psi}}$ by the following formula. 

\begin{equation}\label{G_form}
G_{\pmb{\psi}}(x) := \left\{ \begin{array}{ll}
0, & \text{if} \quad x<v_{\pmb{\psi}(1)};\vspace{0.2cm}\\
\gamma^{1/\pmb{\psi}(n)^*}, & \text{if}\quad x\in [v_{\pmb{\psi}(n)}, v_{\pmb{\psi}(n+1)});\vspace{0.2cm}\\
1, & \text{if} \quad x\geq v_\infty:=\sup\{v_{\pmb{\psi}(n)}\,:\,n\in\GN\}.
\end{array} \right.
\end{equation}
Notice that by Lemma \ref{Fstar} $v_{\pmb{\psi}(n)} \nearrow F_* = (G_{\pmb{\psi}})_*$. 

We want to prove that for every sequence $\{x_n\} \subset \GR$
\[ P\left( M_{\pmb{\psi}(n)} \leq x_n\right) -  G_{\pmb{\psi}}(x_n)^{\pmb{\psi}(n)^*} \longrightarrow 0.\]
It is easy to see that the only nontrivial case is when $x_n \nearrow (G_{\pmb{\psi}})_*$.  For each $n\in\GN$, let $m_n$ be such that $v_{\pmb{\psi}(m_n)} \leq x_n <  v_{\pmb{\psi}(m_n + 1)}$ and let \[ t_1(n) = \frac{\psi_1(n)}{\psi_1(m_n)},\ t_2(n) = \frac{\psi_2(n)}{\psi_2(m_n)},\ \ldots,\ t_d(n) = \frac{\psi_d(n)}{\psi_d(m_n)}.\]
By the monotonicity of $\pmb{\psi}(n)$, for  given $n$ either 
$t_1(n), t_2(n),\ldots, t_d(n) \leq 1$, or $t_i(n) \geq 1$, $i=1,2,\ldots,d$, so that the set $A = \{ \mathbf{t}(n) = (t_1(n), t_2(n),\ldots, t_d(n))\,;\, n\in\GN\}$ satisfies the assumption of part (ii) in Proposition \ref{prop:unif}. Consequently
\begin{align*}
P\left( M_{\pmb{\psi}(n)} \leq x_n\right) &\geq P\left( M_{\pmb{\psi}(n)} \leq v_{\pmb{\psi}(m_n)}\right) \\
&= P(M_{(t_1(n) \psi_1(m_n),\,  t_2(n) \psi_2(m_n), \ldots,\, t_d(n) \psi_d(m_n))}\leq v_{\pmb{\psi}(m_n)})\\
&= \gamma^{t_1(n)\cdot t_2(n)\cdots t_d(n)} + o(1).
\end{align*}
Similarly
\begin{align*}
P\left( M_{\pmb{\psi}(n)} \leq x_n\right) &\leq P\left( M_{\pmb{\psi}(n)} \leq v_{\pmb{\psi}(m_n+1)}\right) \\
&= \gamma^{t_1(n)\cdot t_2(n)\cdots t_d(n)\frac{\pmb{\psi}(m_n)^*}{\pmb{\psi}(m_n+1)^*}} + o(1) = \gamma^{t_1(n)\cdot t_2(n)\cdots t_d(n)} + o(1).
\end{align*}
Therefore 
\[ P\left( M_{\pmb{\psi}(n)} \leq x_n\right) = \gamma^{\frac{\pmb{\psi}(n)^*}{\pmb{\psi}(m_n)^*}} + o(1) = 
G_{\pmb{\psi}}(x_n)^{\pmb{\psi}(n)^*} + o(1),\]
and our claim follows. It remains to replace the purely discontinuous distribution function $G_{\pmb{\psi}}$ with another that is continuous and strictly tail-equivalent to $G_{\pmb{\psi}}$. This can be done following e.g \cite[pp. 703-704]{DJL15}.  

\begin{rem}\label{rem:help}
	Note that so far we have used only the monotonicity of levels $\{v_{\pmb{\psi}}\}$!
\end{rem}

In order to prove the strict tail-equivalence of all $G_{\pmb{\psi}}$ we need a slight improvement of \cite[Proposition 1]{DJL15}.

\begin{lem}\label{lem:key1}
	Let $\{\phi(n)\}\subset \GN$ be increasing and such that 
	$\phi(n)/\phi(n+1) \to 1$. If two distribution functions $G$ and  $H$ satisfy
	\[ \lim_{n\to\infty} G(v_n)^{\phi(n)} =  \lim_{n\to\infty} H(v_n)^{\phi(n)} = \gamma \in (0,1),\]
	for some non-decreasing sequence of levels  $\{v_n\}$, then $G$ and $H$ are strictly tail-equivalent.
\end{lem}
\begin{proof} We mimic \cite[p.701]{DJL15}. Let $x_n \nearrow G_* = H_*$ and let $m_n$ be such that $v_{m_n} \leq x_n < v_{m_n +1}$, $n\in\GN$. Then 
	\[ \phi(m_n) \big( 1 - G(v_{m_n+1})\big) \leq \phi(m_n) \big( 1 - G(x_n)\big) \leq  \phi(m_n) \big( 1 - G(v_{m_n})\big).\]
	Then both $\phi(m_n) \big( 1 - G(v_{m_n})\big) \longrightarrow -\log \gamma$ and 
	\[\phi(m_n) \big( 1 - G(v_{m_n+1})\big) = \frac{\phi(m_n)}{\phi(m_n+1)}\phi(m_n+1) \big( 1 - G(v_{m_n+1})\big)\longrightarrow -\log \gamma,\]
	and so $\phi(m_n) \big( 1 - G(x_n)\big) \longrightarrow -\log \gamma$. But we can repeat this procedure for $H$ equally well.
	Therefore 
	\[ \lim_{n\to\infty} \frac{ 1 - G(x_n)}{ 1 - H(x_n)} = \lim_{n\to\infty}  \frac{\phi(m_n)\big( 1 - G(x_n)\big)}{\phi(m_n)\big( 1 - H(x_n)\big)} = 1.\]
\end{proof}

Let  $G_{\pmb{\psi}'}$ and  $G_{\pmb{\psi}^{\prime\prime}}$ be phantom distribution functions defined by (\ref{G_form}) for monotone curves $\pmb{\psi}'$ and $\pmb{\psi}^{\prime\prime}$. 

By the very definition $G_{\pmb{\psi}'}(v_{\pmb{\psi}'(n)})^{\pmb{\psi}'(n)^*} \longrightarrow \gamma$. So it is enough to show that also 
\[ G_{\pmb{\psi}^{\prime\prime}}(v_{\pmb{\psi}'(n)})^{\pmb{\psi}'(n)^*} \longrightarrow \gamma.\]
Let $m_n$ be such that $\pmb{\psi}^{\prime\prime}(m_n)^* \leq 
\pmb{\psi}'(n)^* < \pmb{\psi}^{\prime\prime}(m_n+1)^*$. Clearly, we have 
\begin{equation}\label{eq:fina} \lim_{n\to\infty} \frac{\pmb{\psi}^{\prime\prime}(m_n)^*}{\pmb{\psi}'(n)^*} = \lim_{n\to\infty} \frac{\pmb{\psi}^{\prime\prime}(m_n+1)^*}{\pmb{\psi}'(n)^*} = 1.
\end{equation}

Since $v_{\bn}$ is {\em strongly} monotone, we have also
$ v_{\pmb{\psi}^{\prime\prime}(m_n)} \leq v_{\pmb{\psi}'(n)}
\leq v_{\pmb{\psi}^{\prime\prime}(m_n+1)}$, hence 
\[ G_{\pmb{\psi}^{\prime\prime}}(v_{\pmb{\psi}^{\prime\prime}(m_n)})^{\pmb{\psi}'(n)^*} \leq G_{\pmb{\psi}^{\prime\prime}}(v_{\pmb{\psi}'(n)})^{\pmb{\psi}'(n)^*}\leq G_{\pmb{\psi}^{\prime\prime}}(v_{\pmb{\psi}^{\prime\prime}(m_n+1)})^{\pmb{\psi}'(n)^*}.\]
By (\ref{eq:fina}) the first and the third terms converge to $\gamma$, and so $G_{\pmb{\psi}'}$ and  $G_{\pmb{\psi}^{\prime\prime}}$ are strictly tail-equivalent. This completes the proof of Theorem \ref{theo1}.

\subsection{Proof of Theorem \ref{th:direct}}

Implication (ii) $\Rightarrow$ (i) is a matter of definitions.
Implication (i) $\Rightarrow$ (iii) can be proved the same way as the necessity in Section \ref{Sec221} (with obvious modifications).

We may also profit from the proof of Theorem \ref{theo1} in the proof of implication (iii) $\Rightarrow$ (ii).
Let $\pmb{\psi}$ be a monotone curve satisfying assumption (iii)  
of Theorem \ref{th:direct}. By Remark \ref{rem:help} function  $G_{\pmb{\psi}}$ defined 
by (\ref{G_form}) is a phantom distribution function for
$\{X_{\bn}\}$  {\em along}  $\pmb{\psi}$. We want to show 
that it is also a phantom distribution function 
for $\{X_{\bn}\}$  along  {\em any other} 
$\pmb{\varphi} \in \mathcal{U}_{\pmb{\psi}}$, i.e. that for 
any $x_n \nearrow (G_{\pmb{\psi}})_* = F_*$ we have
\[ P\left( M_{\pmb{\varphi}(n)} \leq x_n\right) -  G_{\pmb{\psi}}(x_n)^{\pmb{\varphi}(n)^*} \longrightarrow 0.\]  

For each $n\in\GN$, let $m_n$ be such that $v_{\pmb{\psi}(m_n)} \leq x_n <  v_{\pmb{\psi}(m_n + 1)}$ and let \[ t_1(n) = \frac{\varphi_1(n)}{\psi_1(m_n)},\ t_2(n) = \frac{\varphi_2(n)}{\psi_2(m_n)},\ \ldots,\ t_d(n) = \frac{\varphi_d(n)}{\psi_d(m_n)}.\]
We are going to show that the set $A = \{ \mathbf{t}(n) = (t_1(n), t_2(n),\ldots, t_d(n))\,;\, n\in\GN\}$ satisfies the assumption of part (ii) in Proposition \ref{prop:unif}. By the definition of the class $\mathcal{U}_{\pmb{\psi}}$, let $C \geq 1$ be such that
for almost all $n\in\GN$
\begin{equation*}
\pmb{\varphi}(n)\in \bigcup_{j\in\GN} \prod_{i=1}^d [C^{-1} \psi_i(j), C \psi_i(j)].
\end{equation*}
This means that for $n \geq n_0$ there is $j_n \to \infty$ such that
\[ C^{-1} \psi_i(j_n) \leq  \varphi_i(n) \leq C \psi_i(j_n),\quad i=1,2,\ldots, d. \]
Depending on whether $j_n \leq m_n$ or $j_n \geq m_n$ we get that either $t_1(n), t_2(n),\ldots, t_d(n) \leq C$ or  $t_1(n), t_2(n),\ldots, t_d(n) \geq C^{-1}$.
Hence we may apply Proposition \ref{prop:unif} (ii) and we can estimate
\begin{align*}
P\left( M_{\pmb{\varphi}(n)} \leq x_n\right) &\geq P\left( M_{\pmb{\varphi}(n)} \leq v_{\pmb{\psi}(m_n)}\right) \\
&= P(M_{(t_1(n) \psi_1(m_n),\,  t_2(n) \psi_2(m_n), \ldots,\, t_d(n) \psi_d(m_n))}\leq v_{\pmb{\psi}(m_n)})\\
&= \gamma^{t_1(n)\cdot t_2(n)\cdots t_d(n)} + o(1),
\end{align*}
and
\begin{align*}
P\left( M_{\pmb{\varphi}(n)} \leq x_n\right) &\leq P\left( M_{\pmb{\varphi}(n)} \leq v_{\pmb{\psi}(m_n+1)}\right) \\
&= \gamma^{t_1(n)\cdot t_2(n)\cdots t_d(n)\frac{\pmb{\psi}(m_n)^*}{\pmb{\psi}(m_n+1)^*}} + o(1) = \gamma^{t_1(n)\cdot t_2(n)\cdots t_d(n)} + o(1).
\end{align*}
Therefore 
\[ P\left( M_{\pmb{\varphi}(n)} \leq x_n\right) = \gamma^{\frac{\pmb{\varphi}(n)^*}{\pmb{\psi}(m_n)^*}} + o(1) = 
G_{\pmb{\psi}}(x_n)^{\pmb{\varphi}(n)^*} + o(1),\]
and Theorem \ref{th:direct} follows. 

\subsection{Proof of Theorems \ref{theo11} and \ref{th:direct11}}

In view of Theorems \ref{theo1} and \ref{th:direct} and Definition \ref{def:index}, it is enough to show that $F^{\theta}$ is tail equivalent to a phantom distribution function $G$ (resp. a sectorial phantom distribution function $G_{\pmb{\Delta}}$) for $\{X_{\bn}\}$. 

But by (\ref{eq:mainbis11}) we have
\[G(v_{\bn})^{{\bn}^*} \to \gamma_{or},\quad \big(F^{\theta}(v_{\bn})\big)^{{\bn}^*} = \big(F(v_{\bn})^{{\bn}^*}\big)^{\theta} \to 
\gamma_{in}^{\theta} = \gamma_{or},\]
and we can apply Lemma \ref{lem:key1}. The reasoning leading to Theorem \ref{th:direct11} differs only by notation.
\section*{Acknowledgement}
Section 1.4 was performed by I.V. Rodionov with financial support of the Russian Science Foundation under grant No. 19-11-00290.


\begin{thebibliography}{99}
	
	\bibitem{Asmu98}
	Asmussen, S.: Subexponential asymptotics for stochastic processes: extremal behavior, stationary distributions and first passage probabilities. Ann. Appl. Probab. {\bf 8} 354--374 (1998)
	
	\bibitem{BaTa14} Basrak, B., Tafro, A.: Extremes of moving averages and moving maxima on a regular lattice. Probab. Math. Statist. {\bf 34}, 61--79 (2014)
	
	\bibitem{CHOI} Choi, H.: Central Limit Theory and Extremes of Random Fields. PhD dissertation (2002)
	
	\bibitem{DJL15} Doukhan, P., Jakubowski, A., Lang, G.: Phantom distribution functions for some stationary sequences. Extremes {\bf 18}, 697--725 (2015)
	
	\bibitem{Fell70} Feller, W.: An Introduction to Probability Theory and Its Applications.
		Volume II. Second Edition. Wiley, New York (1970).
	
	\bibitem{FePe08} Ferreira, H., Pereira, L.: How to compute the extremal index of stationary random fields. Statist. Probab. Lett. {\bf 78}, 1301--1304 (2008)
	
	\bibitem{GAD} Gadidov, A.: Sectorial convergence of $\,U$-statistics. Ann. Probab. {\bf 33}, 816--822  (2005)
	
	\bibitem{GUT} Gut, A.: Strong laws for independent identically distributed random variables indexed by a~sector. Ann. Probab. {\bf 11}, 569--577  (1983)
	
	\bibitem{Jak91} Jakubowski, A.: Relative extremal index of two stationary processes. Stochastic Process. Appl. {\bf 37}, 281--297 (1991) 
	
	\bibitem{Jak93} Jakubowski, A.: An asymptotic independent representation in limit theorems for maxima of nonstationary random sequences. Ann. Probab. {\bf  21} 819--830 (1993) 
	
	\bibitem{J93} Jakubowski, A.: Asymptotic (r-1)-dependent representation for r-th order statistic from a stationary sequence. Stochastic Process. Appl. {\bf 46} 29--46 (1993). 
	
	\bibitem{JS} Jakubowski, A., Soja-Kukie\l a, N.: Managing local dependencies in asymptotic theory for~maxima of stationary random fields. Extremes.  {\bf 22} 293--315 (2019).
	
	\bibitem{JaTr18} Jakubowski, A., Truszczy\'nski, P.: Quenched phantom distribution functions for Markov chains. Statist. Probab. Lett. {\bf 137} 79--83 (2018)
	
	\bibitem{Lea83} Leadbetter, M.R.: Extremes and local dependence in stationary sequences. Z. Wahrscheinlichkeitstheor. verw. Geb. {\bf 65}, 291--306 (1983)
	
	\bibitem{LLR83} Leadbetter, M.R., Lindgren, G., Rootz\'en, H.: Extremes and Related Properties of Random Sequences and Processes. Springer, New York (1983)
	
	\bibitem{L-R98} Leadbetter, M.R., Rootz\'en, H.: On extreme values in stationary random fields. In: Karatzas, I., Rajput, B.S., Taqqu, M.S. (eds.) Stochastic Processes and Related Topics. In Memory of Stamatis Cambanis \mbox{1943-1995}, pp. 275--285. Birkh\"auser, Boston (1998)
	
	\bibitem{Ling19} Ling, C.: Extremes of stationary random fields on a lattice. Extremes. {\bf 22}, 391--411 (2019).
	
	%\bibitem{OBr74} O'Brien, G.: Limit theorems for the maximum term of a stationary process. Ann. Probab. {\bf 2}, 540--545 (1974) 
	\bibitem{MiYl75} Mittal, Y., Ylvisaker, D.: Limit distributions for the maxima of stationary Gaussian processes. Stochastic Process. Appl. {\bf 3}, 1--18 (1975).
	
	\bibitem{OBr87} O'Brien, G.: Extreme values for stationary and Markov sequences. Ann. Probab. {\bf 15}, \mbox{281--291} (1987)
	
	\bibitem{RRSS06}
	Roberts, G.~O., Rosenthal, J., Segers, J., Sousa, B.: Extremal indices, geometric ergodicity of~Markov chains and MCMC.
	Extremes. {\bf 9} 213--229 (2006).
	
	\bibitem{S-K17} Soja-Kukie\l a, N.: Asymptotics of the order statistics for a process with a regenerative structure. Statist. Probab. Lett. {\bf 131} 108--115 (2017)
	
	\bibitem{Tu06} Turkman, K.F.: A note on the extremal index for space-time processes. J. Appl. Prob. {\bf 43}, 114--126 (2006)
	
	\bibitem{WuSa18} Wu, L., Samorodnitsky, G.: Regularly varying random fields.  arXiv:1809.04477 (2018)
	
\end{thebibliography}
\end{document}